\documentclass[12pt]{amsart}
\usepackage{amssymb,mathtools}
\usepackage{amsmath}
\usepackage[hidelinks]{hyperref}
\usepackage{etoolbox}
\usepackage{enumitem}
\usepackage{xcolor}
\usepackage{tikz-cd}
\usepackage[normalem]{ulem}
\hypersetup{
	colorlinks,
	linkcolor={red!30!black},
	citecolor={blue!50!black},
	urlcolor={blue!80!black}
}

\numberwithin{equation}{section}
\newtheorem{theorem}[equation]{Theorem}
\newtheorem{lemma}[equation]{Lemma}

\newtheorem{corollary}[equation]{Corollary}
\newtheorem{conjecture}[equation]{Conjecture}

\theoremstyle{remark}
\newtheorem{remark}[equation]{Remark}

\newcommand{\p}{\vskip .4cm}

\newcommand{\Z}{\mathbb{Z}}

\newcommand{\cO}{\mathcal{O}}

\newcommand{\til}{\tilde}

\newcommand{\ra}{\rightarrow}
\newcommand{\ira}{\hookrightarrow}
\newcommand{\sra}{\twoheadrightarrow}

\DeclareMathOperator\Lie{Lie}

\DeclareMathOperator\Ad{Ad}
\DeclareMathOperator\ad{ad}

\DeclareMathOperator\Tr{Tr}

\newcommand\Id{\mathrm{Id}}
\newcommand{\fb}{\mathfrak{b}}
\newcommand{\fg}{\mathfrak{g}}
\newcommand{\fh}{\mathfrak{h}}

\newcommand{\fu}{\mathfrak{u}}

\newcommand{\ft}{\mathfrak{t}}

\newcommand{\matr}[1]{\left[\begin{matrix}#1\end{matrix}\right]}

\newcommand{\quo}{/\!/}

\DeclareMathOperator\im{im}

\DeclareMathOperator\SO{SO}
\DeclareMathOperator\Sp{Sp}
\DeclareMathOperator\GL{GL}
\newcommand{\so}{\mathfrak{so}}
\newcommand{\fsp}{\mathfrak{sp}}
\renewcommand{\sp}{\mathfrak{sp}}
\newcommand{\gl}{\mathfrak{gl}}

\newcommand\semisub{{\text{ss}}}
\newcommand\semi{_\semisub}
\newcommand\semis{^\text{ss}}
\newcommand\nilpsub{{\text{nil}}}
\newcommand\nilp{_\nilpsub}

\begin{document}
\title{Jordan decompositions in Lie algebras and their duals}
\begin{abstract}
We provide a discussion of
Jordan decompositions in
the Lie algebra, and the dual Lie algebra, of a reductive group in as uniform a way as possible. We give a counterexample to the claim that Jordan decompositions on the dual Lie algebra are unique, and state an upper bound on how non-unique they can be. We also prove some Chevalley-restriction-type claims about GIT quotients for the adjoint and co-adjoint actions of \(G\).
\end{abstract}

\author{Loren Spice}
\address{2840 W.~Bowie St., Fort Worth, TX 76109}
\email{l.spice@tcu.edu}

\author{Cheng-Chiang Tsai}
\address{Institute of Mathematics, Academia Sinica, 6F, Astronomy-Mathematics Building, No. 1,
Sec. 4, Roosevelt Road, Taipei, Taiwan \vskip.2cm
also Department of Applied Mathematics, National Sun Yat-Sen University, and Department of Mathematics, National Taiwan University}

\email{chengchiangtsai@gmail.com}

\maketitle

\section{Introduction}

The notion of a Jordan decomposition in a Lie algebra is
well known among those who work with affine algebraic groups over fields.
It allows one, e.g., to reduce the study of centralisers of
arbitrary elements of a Lie algebra to centralisers of
semisimple elements, which are well behaved, and centralisers of
nilpotent elements.
Less well known is the concept of a Jordan decomposition on
a \emph{dual} Lie algebra.
This concept seems to require some restriction on the group whose
dual Lie algebra we are considering---%
the one that we impose throughout is reductivity---%
and even then some subtlety is involved.
See Remark \ref{rem:dual-Jordan} for the necessary definitions
(of semisimplicity, nilpotence, and the Jordan decomposition).

Past literature has claimed that, in a setting slightly more
restrictive than ours,
Jordan decompositions on the dual Lie algebra exist and are unique.
The existence claim is correct, but we provide a
counterexample, or rather a family of counterexamples,
to uniqueness.  See \S\ref{sec:nonuni}.
We thus regard it as useful to give a short discussion of
Jordan decompositions on the Lie algebra and the dual Lie algebra of
a reductive group in as uniform a way as possible.
This means that, though we might have to handle different
special cases differently for the two situations, we have
made statements that are the same for the Lie algebra and
the dual Lie algebra whenever possible.
We state our main result as Theorem \ref{thm:main} below.
It is a pleasure here to acknowledge our clear debt to \cite{KW76},
on which our work builds, while
correcting some statements in \cite[\S3.8 and Theorem 4]{KW76} that
are only guaranteed to be correct in large characteristic.

Theorem \ref{thm:main}\ref{uni} places an upper bound on
``how non-unique'' a Jordan decomposition can be.
This upper bound is sufficient for some purposes, but is not optimal.
See Theorem \ref{thm:uni} for a sharper bound.

Note parts \ref{iso2}, \ref{last}.
Although Chevalley restriction, and its ``dual'' analogue, can
fail in characteristic \(2\), these results show how to repair it.

\begin{theorem}\label{thm:main}
Let \(G\) be a connected reductive group over
an algebraically closed field \(k\).
Write \(\fg\) and \(\fg^*\) for the
Lie algebra and dual Lie algebra of \(G\), or
the associated vector groups.

When $V$ is an affine variety on which $G$ acts, we write $V\quo G$
for the GIT quotient $\operatorname{Spec}\cO(V)^G$.
\begin{enumerate}[label=(\roman*)]
    \item\label{key} A $G$-orbit in $\fg$ or $\fg^*$ is closed if and only if it consists of semisimple elements.
    \item\label{Jordan} Every element of \(\fg\) or \(\fg^*\) has a Jordan decomposition (though a Jordan decomposition in $\fg^*$ might not be unique; see Section \ref{sec:nonuni}).
    \item\label{uni} For every element \(X\) of \(\fg\) or \(\fg^*\), there is a unique closed orbit in the closure of the \(G\)-orbit of \(X\), and this closed orbit contains the semisimple part \(X\semi\) of every Jordan decomposition \(X = X\semi + X\nilp\).
    \item\label{ss} Consider the closures
$\overline{\fg\semis}$ and $\overline{(\fg^*)\semis}$ of
the constructible subsets of semisimple elements in
\(\fg\) and \(\fg^*\).
For every maximal torus \(T \subset G\), we have that
\(\overline{\fg\semis}\) is the sum of
\(\ft\) and
all root spaces for \(T\) in \(\fg\) corresponding to
roots that are non-zero in \(\ft^*\); and
\(\overline{(\fg^*)\semis}\) is the sum of
\((\fg^*)^T\) and
all weight spaces for \(T\) in \(\fg^*\) corresponding to
roots \(\alpha\) for which the corresponding coroots
\(h_\alpha := d\alpha^\vee(1)\) are non-zero in \(\ft\).
If the characteristic of \(k\) is not \(2\), then
\(\overline{\fg\semis}\) equals \(\fg\) and
\(\overline{(\fg^*)\semis}\) equals \(\fg^*\).
If the characteristic of \(k\) equals \(2\), then \(\overline{\fg\semis}\) is the sum of the Lie algebras of all direct factors of \(G\) that are not of type \(\Sp_{2n}\) for some positive integer \(n\), with a canonical \(\so_{2n}\) summand for each \(\Sp_{2n}\) direct factor; and \(\overline{(\fg^*)\semis}\) is the sum of the dual Lie algebras of all direct factors of \(G\) that are not of type \(\SO_{2n + 1}\) for some positive integer \(n\), with a canonical \(\so_{2n}^*\) direct summand for each \(\SO_{2n + 1}\) direct factor.
%If the characteristic of \(k\) equals \(2\), write $G=G_{\Sp}\times\prod\Sp_{2n_i}$ where $G_{\Sp_{2n}}$ is the maximal direct factor of $G$ that does not have $\Sp_{2n}$ for any $n\in\Z_{\ge 1}$ as a direct factor. Then \(\overline{\fg\semis}\) is the sum of $\Lie G_{\Sp}$ and a canonical \(\so_{2n}\) summand for each \(\Sp_{2n}\) direct factor.
%Likewise, $G=G_{\SO_{\mathrm{odd}}}\times\prod\SO_{2n_i+1}$ where $G_{\SO_{\mathrm{odd}}}$ is the maximal direct factor of $G$ that does not have $\SO_{2n+1}$ for any $n\in\Z_{\ge 1}$ as a direct factor. Then ; and \(\overline{(\fg^*)\semis}\) is the sum of $\Lie^* G_{\SO_{2n+1}}$ with a canonical \(\so_{2n}^*\) direct summand for each \(\SO_{2n + 1}\) direct factor.
    \item\label{iso2} For every maximal torus $T\subset G$, if we put $W=N_G(T)/T$, then the natural map $\ft\quo W\ra\overline{\fg\semis}\quo G$ (respectively, $\ft^*\quo W\ra\overline{(\fg^*)\semis}\quo G$) is
%bijective on $k$-points.
%    \item\label{iso2} In fact, the maps in \ref{iso} are isomorphisms.
an isomorphism.
    \item\label{last} The natural map $\overline{\fg\semis}\quo G\ra\fg\quo G$ (respectively, $\overline{(\fg^*)\semis}\quo G\ra\fg^*\quo G$) is bijective on $k$-points.
\end{enumerate}
\end{theorem}

There are two natural ways that
we might hope to strengthen Theorem \ref{thm:main}.
We state them as Conjecture \ref{conj:main} to draw attention to them, though Conjecture \ref{conj:main}(ii) might be questionable.
Beyond Theorem \ref{thm:main}, we do not
make progress in this paper towards
the proof (or disproof) of Conjecture \ref{conj:main}.

\begin{conjecture}\label{conj:main}\hfill
\begin{enumerate}[label=(\roman*)]
    \item\label{iso3} The map \(\overline{\fg\semis}\quo G\ra\fg\quo G\) (respectively, \(\overline{(\fg^*)\semis}\quo G\ra\fg^*\quo G\)) in Theorem \ref{thm:main}\ref{last} is a universal homeomorphism.
    \item\label{affine} The GIT quotients $\fg\quo G$ and $\fg^*\quo G$ are affine spaces.
\end{enumerate}
\end{conjecture}

We provide a brief overview of our paper.
In Section \ref{sec:prep}, we fix some notation and definitions,
and use geometric reductivity to prove
Theorem \ref{thm:main}\ref{Jordan} and \ref{last}, and
Lemma \ref{lem:orbit}, which is a weaker version of
Theorem \ref{thm:main}\ref{uni} (and will allow us to deduce the latter from
Theorem \ref{thm:main}\ref{key}).
In Section \ref{sec:Jordan-rems}, we
make a few remarks justifying our definitions
(which are those of \cite{KW76})
of semisimple and nilpotent elements, and Jordan decompositions, and
comparing them to those of \cite{CM93}.
In Section \ref{sec:reduction}, we explain why,
to complete the proof of Theorem \ref{thm:main}, it
suffices to consider
one ``bad'' family of adjoint actions, and
one ``bad'' family of co-adjoint actions, both
in characteristic \(2\).
(The ``bad'' family of co-adjoint actions is precisely the one
identified as problematic, and so excluded, in \cite{KW76}.)
We include two arguments to handle these families,
in Sections \ref{sec:first} and \ref{sec:second},
because we think that they provide interestingly different perspectives.
In Section \ref{sec:nonuni}, we give a
small family of counterexamples of
Jordan decompositions in \(\fg^*\) that are not unique; and then,
in Section \ref{sec:uni}, we upgrade
Theorem \ref{thm:main}\ref{uni} to Theorem \ref{thm:uni}, which
shows that the counterexamples in Section \ref{sec:nonuni} are
the only possibilities.
Theorem \ref{thm:uni} recovers the uniformity lost
in our statement of
Theorem \ref{thm:main}\ref{Jordan}, in the sense that it is
a result on \(\fg^*\) that describes the possible
non-uniqueness of Jordan decompositions there, but
whose analogue on \(\fg\) implies that Jordan decompositions there are
unique.
Finally, in
Section \ref{sec:nilp-or-semi-is-nilp-or-semi}, we show that
the only Jordan decompositions of
semisimple or nilpotent elements are
the obvious, trivial ones, so that the non-uniqueness of
Section \ref{sec:nonuni} is not an issue in this special case.

\section{Notation and preparation}\label{sec:prep}

The following notation is mostly as in
Theorem \ref{thm:main}, but we introduce it here so as to be able to
refer to it throughout the paper.

Let $k$ be an algebraically closed field. Let $G$ be a connected reductive group over $k$, $\fg$ its Lie algebra, and $\fg^*$ its dual Lie algebra.
In general, if we denote a group by an uppercase Roman letter, then
we use the corresponding lowercase Fraktur letter for its Lie algebra, so,
for example, once we have groups \(T\) and \(B\), we will
denote their Lie algebras by \(\ft\) and \(\fb\).
We may also write \(\Lie(\cdot)\) in place of the Fraktur notation when
it improves readability.

When $V$ is an affine variety on which $G$ acts, we write $V\quo G$
for the GIT quotient $\operatorname{Spec}\cO(V)^G$.
We will sometimes also use \(\fg\) and \(\fg^*\) for the associated vector groups, so that the notations \(\fg\quo G\) and \(\fg^*\quo G\) make sense.

By elements and points in a scheme over $k$ we refer to the $k$-points unless otherwise specified. Centralizers (such as $C_G(X)$ for an element $X$ of $G$, $\fg$ or $\fg^*$) and centers (such as $Z(G)$) will however always mean the schematic centralizers and centers.
%By a $G$-orbit of some element(s) in a $G$-scheme (always separated of finite type) over $k$ we refer to the reduced locally closed subscheme, which we sometimes identify with its $k$-points when there is no danger of confusion.

Fix a maximal torus \(T\subset G\).
	%and write \(\ft\) for its Lie algebra.
The inclusion of the \(T\)-fixed subspace \((\fg^*)^T\) of \(\fg^*\) is a splitting of the quotient \(\fg^*\sra\ft^*\), which we regard as an inclusion of \(\ft^*\) in \(\fg^*\).
Then \(\ft^*\) is the subspace of \(\fg^*\) that
vanishes on the unique \(T\)-stable complement to
\(\ft\) in \(\fg\).

Write \(\Phi(G, T)\) for the set of roots of \(T\) in \(\fg\), and
\(\Phi^\vee(G, T)\) for the dual root system.
For every \(\alpha \in \Phi(G, T) \subset X^*(T)\), we have
an element
\(\alpha^\vee \in \Phi^\vee(G, T)\) of
\(X_*(T)\), and
an element \(h_\alpha := d\alpha^\vee(1)\) of \(\ft\)
(which is identified \textit{via} the evaluation isomorphism
\(X_*(T) \otimes_{\mathbb Z} k \ra \ft\) with
\(\alpha^\vee \otimes_{\mathbb Z} 1\)).
We will call both `coroots', and rely on context, or
explicit use of the notation \(\alpha^\vee\) or \(h_\alpha\), to clarify.

Let \(B\) be a Borel subgroup of \(G\) containing \(T\) and \(U\) the unipotent radical of \(B\), and
write \(\fb\) and \(\fu\) for their Lie algebras.
%There are unique \(T\)-equivariant embeddings of \(\fu^*\) in \(\fb^*\) and of \(\fb^*\) in \(\fg^*\), which we use to regard them as subspaces.
%Specifically, if we write
%\(B'\) for the Borel subgroup of \(G\) that is
%opposite to \(B\) with respect to \(T\), and
%\(U'\) for its unipotent radical, then we identify
%\(\fb^*\) with \((\fg/\fu')^*\) and
%\(\fu^*\) with \((\fg/\fb')^*\).
%Note that these identifications usually depend on \(T\).

%An element of \(\fg\) is semisimple, respectively nilpotent, if and only if it is \(G(k)\)-conjugate to an element of \(\ft\), respectively \(\fu\) \cite[Proposition 11.8 and proof of Proposition 14.26]{Bo}.
%Since restriction to \(\fu\) provides a \(B\)-equivariant isomorphism of
%\((\fg/\fb)^*\) onto \(\fu^*\), a reasonable analogy would be to call
%an element of \(\fg^*\) semisimple, respectively nilpotent, if and only if it is \(G(k)\)-conjugate to an element of \(\ft^*\), respectively \((\fg/\fb)^*\).
%Indeed, this is the definition in \cite[\S3, p.~140]{KW76}.

Remark \ref{rem:dual-Jordan} re-states
the definition \cite[\S3, p.~140]{KW76} of a Jordan decomposition in \(\fg^*\)
in an equivalent form that is the one that
we will use throughout the paper.

\begin{remark}
\label{rem:dual-Jordan}
Let \(X\) be an arbitrary element of \(\fg^*\).
By definition, an element \(X\semi \in \fg^*\) is
semisimple if and only if
\(C_G(X\semi)^\circ\) contains a maximal torus in \(G\).
Then \cite[Lemma 3.1ii)]{KW76} says that
\(X\semi\) is the semisimple part of
a Jordan decomposition of \(X\) if and only if, for some
choice of maximal torus \(T\) in \(C_G(X\semi)^\circ\) and
Borel subgroup \(B\) of \(G\) containing \(T\),
we have that
\(X\) agrees with \(X\semi\) on \(\fb\),
	%where \(\fb\) is the Lie algebra of \(B\),
and
\(X\) is trivial on the unique \(T\)-stable complement to
\(\Lie(C_G(X\semi)^\circ)\) in \(\fg\).
(Note that the complement might
depend on the choice of \(T\); see
\cite[Remark 3.2]{Spi21}.)

In this perspective, we say that
an element \(X\) is nilpotent if and only if
it has a Jordan decomposition with trivial semisimple part,
i.e., if and only if
there is some Borel subgroup \(B\) of \(G\) such that
\(X\) vanishes on \(\fb\).
\end{remark}

Since the notion of a Jordan decomposition on the Lie-algebra side is defined for any algebraic group, not just a reductive one, we may apply \cite[Proposition 11.8]{Bo} with \(B\) in place of \(G\) to conclude that
every element of \(\fb\) admits a Jordan decomposition whose semisimple part is tangent to some maximal torus in \(B\), hence belongs to some \(B(k)\)-conjugate of \(\ft\).
The analogous statement for \(\fg^*\), that
every element of \((\fg/\fb)^*\) has a Jordan decomposition whose semisimple part belongs to some \(B(k)\)-conjugate of \(\ft^*\), is proven in
\cite[\S3.8]{KW76}
by an abstract computation with root systems and Chevalley constants that
carries through unchanged even when we drop the restriction in
\cite[\S3]{KW76} that \(G\) be almost simple, and
not of type \(\SO_{2n + 1}\) if the characteristic of \(k\) equals \(2\).
By \cite[Proposition 14.25]{Bo} for \(\fg\), respectively \cite[Proposition 3.3]{Spi21} for \(\fg^*\), every orbit of \(G\) meets \(\fb\), respectively \((\fg/\fb)^*\), so we obtain the analogous fact for the orbits of arbitrary elements of \(\fg\), respectively \(\fg^*\).
%every element has a Jordan decomposition $X=X\semi+X\nilp$ (not unique in general), which in particular implies that for every orbit $\Ad^*(G)X\subset\fg^*$, its closure contains a semisimple orbit $\Ad^*(G)X\semi$.
This proves Theorem \ref{thm:main}\ref{Jordan}.
If \(X\) is an element of \(\fg\) or \(\fg^*\) with
Jordan decomposition \(X = X\semi + X\nilp\), then
\(X\semi\) is contained in the closure of the \(G\)-orbit of \(X\)
(by \cite[Proposition 3.8]{Spi21} for \(\fg^*\), and
an analogous argument for \(\fg\)), so
a closed orbit of \(G\) in \(\fg\) or \(\fg^*\) already
consists of semisimple elements.

Since $G$ is reductive, and therefore geometrically reductive
\cite[Theorem 5.2]{Hab75},
we have for every finite-type, affine \(G\)-variety \(V\) that \(V\quo G\) is again of finite type \cite[p.~372, Main Theorem]{Nag64}; that every orbit closure contains a unique closed orbit; and that closed orbits through points in \(V(k)\) correspond bijectively to points in \((V\quo G)(k)\). This proves Theorem \ref{thm:main}\ref{last}, and that \ref{uni}
%and \ref{iso}
follows from \ref{key}.
It also proves a weaker version of
Theorem \ref{thm:main}\ref{iso2}, namely, that
the maps occurring there are
bijective on \(k\)-points.
Our discussion has also established Lemma \ref{lem:orbit}, which is
a weaker version of Theorem \ref{thm:main}\ref{uni}.

\begin{lemma}\label{lem:orbit} Every closure of a \(G\)-orbit in $\fg^*$ contains a unique closed orbit, which consists of semisimple elements.
\end{lemma}

\section{Remarks on Jordan decompositions}
\label{sec:Jordan-rems}

We hope that the reader will agree that
the notion of a Jordan decomposition on \(\fg\),
especially its extremely strong connection to
the classical Jordan decomposition in matrix algebras \(\gl_n\) coming from
a choice of faithful representation \(G \ra \GL_n\),
is sensible.
Granting that, Remark \ref{rem:Killing} argues that
the notion of a Jordan decomposition on \(\fg^*\) is also sensible
because any \(G\)-equivariant identification of \(\fg\) with \(\fg^*\)
identifies it with the notion of a Jordan decomposition on \(\fg\).

\begin{remark}
\label{rem:Killing}
Theorem \ref{thm:main}\ref{uni} shows that
the semisimple elements of \(\fg^*\) are
precisely those whose \(G\)-orbits are closed, and
Remark \ref{rem:nilp-is-nilp} below
(and \cite[Proposition 3.8]{Spi21}) show that
the nilpotent elements of \(\fg^*\) are
precisely those that contain \(0\) in
the closure of their \(G\)-orbits.
The analogues for \(\fg\) of these two facts are
\cite[Corollary 3.6]{SS70} and
\cite[Theorem 10.6(4) and Proposition 14.25]{Bo}.
(Actually, \cite[Corollary 3.6]{SS70} is stated only for
semisimple groups.
We were not able to find a reference for reductive groups other than
our Theorem \ref{thm:main}\ref{key}.)
In particular, if there is a \(G\)-equivariant isomorphism
of \(\fg\) with \(\fg^*\), then
it identifies the semisimple, respectively nilpotent,
elements in the two sets.
Further, if \(X\semi \in \fg\) is semisimple with image
\(X^*\semi \in \fg^*\), we have that
\(C_G(X\semi)\) equals \(C_G(X^*\semi)\), and,
for every maximal torus \(T\) fixing \(X\semi\),
the spaces \(\Lie(C_G(X\semi))\) and
\((\fg/{{}_T{\Lie(C_G(X^*\semi))^\perp}})^*\) are identified, where
\({}_T{\Lie(C_G(X^*\semi)^\circ)^\perp}\) is
the unique \(T\)-stable complement to
\(\Lie(C_G(X^*\semi)^\circ)\) in \(\fg\).
This shows that such an isomorphism also
identifies Jordan decompositions of elements of
\(\fg\) and \(\fg^*\).
\end{remark}

Remark \ref{rem:KW-vs-CM} discusses the compatibility (or in-)
of our definition of semisimplicity and nilpotence with
a definition that is sometimes used in characteristic \(0\).

\begin{remark}
\label{rem:KW-vs-CM}
In \cite[\S1.3]{CM93}, alternate definitions are given for
semisimplicity and nilpotence in \(\fg^*\).
Those definitions are given only over \(\mathbb C\); here
we discuss to what extent they carry over to arbitary fields, and
their relation to our notions.

The definition of semisimplicity involves some choices that are
particular to the case where \(k\) has characteristic \(0\), for example
a close identification between groups and Lie algebras, and
the fact that, in that case, all affine algebraic \(k\)-groups are smooth.
One reasonable characteristic-free re-interpretation of it is that
\(X^*\) is semisimple if and only if
\(C_G(X^*)^\circ\) is (smooth and) reductive.
Semisimplicity in the sense of \cite{KW76} implies
semisimplicity in this sense, by \cite[Lemma 3.1iii)]{KW76}.
We do not know whether the reverse implication holds.

The definition of nilpotence in \cite[\S1.3]{CM93} is that
an element \(X^* \in \fg^*\) is
nilpotent if and only if it is trivial on
\(\Lie(C_G(X^*))\).
In small characteristic, nilpotence in the sense of \cite{KW76} can
fail to imply nilpotence in the sense of \cite{CM93},
and \textit{vice versa}.
For example, suppose the characteristic of \(k\) is \(2\).
The element
\(X^* : \begin{pmatrix} a & b \\ c & -a \end{pmatrix} \mapsto c\)
of \(\mathfrak{sl}_2^*\) is nilpotent in the sense of \cite{KW76}; but
\(\Lie(C_G(X^*))\) equals \(\mathfrak{sl}_2\), while
\(X^*\) is not trivial on all of \(\Lie(C_G(X^*))\), and hence
not nilpotent in the sense of \cite{CM93}.
On the other hand, the element
\(X^* : \begin{bmatrix} a & b \\ c & d \end{bmatrix} \mapsto a + c + d\)
of \(\mathfrak{pgl}_2^*\) has the property that
\(\Lie(C_G(X^*))\) is
the strictly upper-triangular subalgebra of \(\fg\), on which
\(X^*\) is trivial, so that \(X^*\) is nilpotent in the sense of \cite{CM93};
but
\(X^*\) has a Jordan decomposition with semisimple part
the trace map, and hence,
by Remark \ref{rem:nilp-is-nilp} below,
is not nilpotent in the sense of \cite{KW76}.
%(In fact, Theorem \ref{thm:main}\ref{uni} shows that
%the Jordan decomposition is unique in this case.)
\end{remark}

\begin{lemma}
\label{lem:CM-nilp-is-KW-nilp}
If, for every full-rank, reductive subgroup \(H\) of \(G\),
restriction to \(\Lie(Z(H))\) provides an injection
\((\fh^*)^H \ra \Lie(Z(H))\),
%(for example, by the proof of \cite[Lemma 4.1.5]{FKS23}, if
%\(G\) is simply connected and the characteristic of \(k\) is not \(2\)),
	%% This gives it for \(G\), but not for
	%% all full-rank subgroups.
then
nilpotence in the sense of \cite{CM93} implies
nilpotence in the sense of \cite{KW76}.
\end{lemma}

\begin{proof}
Suppose that
\(X^* \in \fg^*\) is nilpotent in the sense of \cite{CM93}.
By Theorem \ref{thm:main}\ref{Jordan} and
Remark \ref{rem:dual-Jordan}, there are
a maximal torus \(T\) in \(G\),
a Borel subgroup \(B\) of \(G\) containing \(T\), and
an element
\(X^*\semi \in \ft^*\) such that
\(X^*\) agrees with \(X^*\semi\) on \(\fb\), and,
if we put \(H = C_G(X^*\semi)^\circ\), then
\(X^*\) is trivial on the unique \(T\)-stable complement
\({}_T\fh^\perp\) to \(\fh\) in \(\fg\).
Since restriction to \(\fh\) provides
an \(H\)-equivariant isomorphism
\((\fg/{}_T\fh^\perp)^* \ra \fh^*\), we have in particular that
\(X^*\) is fixed by \(Z(H)\), hence annihilated by \(\Lie(Z(H))\).
That is, \(\Lie(Z(H))\) is contained in \(\Lie(C_G(X^*))\), hence in
\(\ker(X^*)\); but
\(\Lie(Z(H))\) is also contained in \(\ft\), hence in \(\fb\), so
\(X^*\semi\) agrees with \(X^*\) on \(\Lie(Z(H))\), hence
vanishes there.
Since \(X^*\semi\) belongs to \((\fh^*)^H\), it is
trivial by our injectivity assumption, so \(X^*\) is nilpotent
in the sense of \cite{KW76}.
\end{proof}

Remark \ref{rem:from-full-rank} shows that
Jordan decompositions may be inflated to \(\fg^*\) from
the dual Lie algebras of
full-rank, reductive subgroups of \(G\).

\begin{remark}
\label{rem:from-full-rank}
Let \(H\) be a full-rank, reductive subgroup of \(G\),
\(X^*_H\) an element of \(\fh^*\), and
\(X^*_H = X^*_{H\,\semisub} + X^*_{H\,\nilpsub}\)
a Jordan decomposition in \(\fh^*\).
Then there are a maximal torus \(T\) in \(H\) and
a Borel subgroup \(B_H\) of \(H\) containing \(T\) such that
\(X^*_H\) agrees with \(X^*_{H\,\semisub}\) on \(\fb_H\),
\(X^*_{H\,\semisub}\) is fixed by \(T\), and
\(X^*_H\) is trivial on
the unique \(T\)-stable complement
\({}_T{\Lie(C_H(X^*_{H\,\semisub}))}^\perp\) to
\(\Lie(C_H(X^*_{H\,\semisub}))\) in
\(\fh\).
Write \({}_T\fh^\perp\) for the unique \(T\)-stable complement to
\(\fh\) in \(\fg\), and
\(X^*\) for the extension of \(X^*_H\) trivially across \({}_T\fh^\perp\); and
similarly for \(X^*\semi\) and \(X^*\nilp\).
Let \(B\) be a Borel subgroup of \(G\) containing \(B_H\).
Then \(X^*\semi\) is fixed by \(T\);
since \(\fb_H + {}_T\fh^\perp\) contains \(\fb\), we have that
\(X^*\) agrees with \(X^*\semi\) on \(\fb\); and,
since \({}_T{\Lie(C_H(X^*_{H\,\semisub}))}^\perp + {}_T\fh^\perp\) contains
the unique \(T\)-stable complement to
\(\Lie(C_G(X^*\semi))\) in \(\fg\), we have that
\(X^*\) is trivial on that complement.
That is, \(X^* = X^*\semi + X^*\nilp\) is
a Jordan decomposition of \(X^*\).
\end{remark}

\section{Reduction to $\sp_{2n}$ and $\so_{2n+1}^*$}\label{sec:reduction}

We show that Theorem \ref{thm:main} and Conjecture \ref{conj:main}\ref{iso3} can be reduced to the case where
\(k\) has characteristic \(2\), and
we are considering either
the adjoint action of \(\Sp_{2n}\)
(which we shall denote by $\Sp_{2n}\curvearrowright\sp_{2n}$) or
the coadjoint action of \(\SO_{2n + 1}\)
(which we shall denote by $\SO_{2n+1}\curvearrowright\so_{2n+1}^*$)
for some positive integer \(n\).

As observed in \cite[p.~83]{Jan04},
the proof in \cite[II.3.17\textquotesingle]{SS70} that
\(\ft\quo W \ra \fg\quo G\) is
an isomorphism, stated there under the assumption that
\(G\) is adjoint, actually requires only that
no root of \(T\) on \(\fg\) is \(0\) when viewed as an element of \(\ft^*\).
We will show how to handle the opposite case at the end of this section.

The results of \cite[Theorem 4]{KW76} are stated under the assumptions that
\(G\) is almost simple (imposed in \cite[\S1, p.~136]{KW76}), and that
\(p\) does not equal \(2\) or
\(G\) is not of type \(\SO_{2n + 1}\) for any positive integer \(n\)
(imposed in \cite[\S3, p.~140]{KW76}).
For any reductive group \(G\), there is an isogeny onto \(G\) from
the product of the central torus of \(G\) and its almost-simple factors, but
the derivative can have a kernel for small \(p\).
Thus, it is not a formal argument to remove the first assumption, but
we describe how to carry it out below.
Removing the second requires a change to the statement of
\cite[Theorem 4]{KW76}.
In fact, we have counterexamples to some of
the intermediate results of \cite{KW76} even in
their setting.
We describe the situation below.

For every \(\alpha \in \Phi(G, T)\), we have an element
\(h_\alpha := d\alpha^\vee(1)\) of \(\ft\).
Suppose for the moment that all these elements are non-zero.
This is precisely what is needed to ensure that
the set of
regular (in \(\fg^*\)) semisimple elements of \(\ft^*\), i.e.,
semisimple elements that do not lie on any coroot hyperplane,
denoted by \(\Omega\) in \cite[\S3, p.~140]{KW76}, is
dense and open in \(\ft^*\).
Then the arguments of \cite[\S3]{KW76} go through unchanged, with
two exceptions.
First, the statement in \cite[Lemma 3.1i)]{KW76} that
\(\Phi(C_\fg(X^*), T)\) is closed
in \(\Phi(G, T)\) for every \(X^* \in \ft^*\)
can fail in small characteristic.
Second,
the claim in \cite[\S3.8i)]{KW76} that
a certain map from
a quotient of a Borel subgroup to
affine space is an isomorphism
can also fail.
See Section \ref{sec:nonuni} for an example of this behaviour
(with \(G = \Sp_4\) in characteristic \(2\)), where we also show that
the containment of connected centralizers asserted in
\cite[\S3.8ii)]{KW76}, and
the uniqueness asserted in \cite[Theorem 4iv)]{KW76}, can
fail.

Thus, we must remove uniqueness in
\cite[Theorem 4iv)]{KW76}, and discard
\cite[\S3.8]{KW76} entirely.
Otherwise, aside from \cite[\S3.8]{KW76} and \cite[Theorem 4iv)]{KW76},
the rest of the proof of \cite[Theorem 4]{KW76} is valid
with no restriction on \(G\) beyond
every coroot \(h_\alpha\) being non-zero.

We are left with the case when there is some
\(\alpha \in \Phi(G, T)\) such that \(h_\alpha \in \ft\) equals \(0\), and thus the image \(\chi(h_\alpha)\) in $k$ of $\langle\chi,\alpha^{\vee}\rangle$ is $0$ for every $\chi\in X^*(T)$.
In particular, since \(\langle\alpha, \alpha^\vee\rangle\) equals \(2\),
the characteristic of \(k\) is $2$; and
there is no root \(\beta\) such that
the integer \(\langle\beta, \alpha^\vee\rangle\) is odd.
By inspection of the Dynkin diagrams,
this means that
the component of \(\Phi(G, T)\) containing \(\alpha\) is
of type \(\mathsf B_n\) for some \(n\), possibly \(n = 1\).
If the corresponding almost-simple factor $G'$ of \(G\) were
simply connected, then $\alpha\in\Phi(G',T')$ (where $T':=T\cap G'$) would not be divisible by $2$ in $X_*(T')$ since any root is indivisible in the root lattice.
In particular, \(h_\alpha\) would be non-zero in \(\fg'\ira\fg\), a contradiction.
%and we chose a system of simple roots for that simple factor containing \(\alpha\), then the half-sum \(\rho\) of the corresponding system of positive roots would lie in \(X^*(T)\) and satisfy \(\langle\rho, \alpha^\vee\rangle = 1\), which is a contradiction.
Thus the factor is actually adjoint, i.e., is
a \textbf{direct} factor of \(G\) of type \(\operatorname{SO}_{2n + 1}\).
%This happens iff $\alpha^{\vee}$ is divisible by $\operatorname{char}(k)$ in $X_*(T)$, which happens iff $\operatorname{char}(k)=2$ and $\frac{1}{2}\alpha^{\vee}\in X_*(T)$.
%Viewing $\Phi^{\vee}(G,T)$ as a root system, it is the product of a finite number of irreducible root systems. Consider $\Psi^{\vee}\subset\Phi^{\vee}(G,T)$ those coroots contained in the irreducible root system that contains $\alpha^{\vee}$, and $\Psi\subset\Phi(G,T)$ the corresponding roots. Let
%\[\Lambda:=\langle\Psi\rangle\subset X^*(T),\;\;\Lambda^{\vee}:=\{\lambda\in X_*(T)\otimes\Q\;|\;(\lambda,\alpha)\in\Z,\;\forall\alpha\in\Phi(G,T)\}.\]
%It is well-known that $\Lambda$ and $\Lambda^{\vee}$ are dual to each other, so that $(\Lambda,\Lambda^{\vee},\Psi,\Psi^{\vee})$ is a root datum corresponding to an almost simple group. We have $\frac{1}{2}\alpha^{\vee}\in\Lambda^{\vee}$. By inspecting the classification of root data associated to almost simple groups, we see that the corresponding almost simple group must be of the form $\SO_{2n+1}$, $n\in\Z_{>0}$.

We now show that, similarly, the failure of the
Chevalley restriction theorem (for \(\fg\)) can be attributed to
the presence of symplectic-group direct factors.
Thus, suppose instead that there is some
\(\alpha \in \Phi(G, T)\) whose differential
\(d\alpha\), viewed as an element of \(\ft^*\), equals \(0\).
If we write \(G^*\) for
the dual group of \(G\) (viewed as a \(k\)-group, not a \(\mathbb C\)-group) and
\(T^*\) for the maximal torus of \(G^*\) corresponding to \(T\), then
the preceding argument shows that
there is some positive integer \(n\) such that
\(G^*\) admits \(\SO_{2n + 1}\) as a direct factor with
\(\alpha \in \Phi^\vee(\SO_{2n + 1}, T^*)\).
It follows that
\(G\) admits \(\SO_{2n + 1}^* = \Sp_{2n}\) as a direct factor with
\(\alpha \in \Phi(\Sp_{2n}, T)\).

\section{First approach}
\label{sec:first}

Suppose \(k\) has characteristic \(2\).
We work with the case $G=\SO_{2n+1}\curvearrowright\so_{2n+1}^*$; the case of $\Sp_{2n}\curvearrowright\sp_{2n}$ is almost identical except for being easier at some places.

%Now $G=\Sp_{2n}$ or $G=\SO_{2n+1}$, and $k$ has characteristic $2$. In the former case $V=\fsp_{2n}$ and in the latter case $V=\so_{2n+1}^*$. In fact we have an isogeny $\SO_{2n+1}\sra\Sp_{2n}$ that is bijective on $k$-points, so we are sort of dealing with the same group in the two cases.
As in Section \ref{sec:prep},
fix a maximal torus \(T\) in \(\SO_{2n + 1}\), and
identify \(\ft^*\) with
the \(T\)-fixed subspace of $\so_{2n+1}^*$. There is an isogeny $f:\SO_{2n+1}\ra\Sp_{2n}$ that induces an isomorphism of \(T\) onto a maximal torus in \(\Sp_{2n}\), and sends the short roots of $\Sp_{2n}$ to the long roots of $\SO_{2n+1}$, but sends the long roots of $\Sp_{2n}$ to twice the short roots of $\SO_{2n+1}$, in the sense of \cite[Theorem 9.6.5]{Spr09}.

If \(v_0 \in \ft^*\) lies in the complement of
all fixed subspaces of Weyl elements that
do not act trivially on \(\ft^*\), and satisfies
\(v_0(h_\alpha) \ne 0\) for all \(\alpha \in \Phi(\SO_{2n + 1}, T)\) such that
\(h_\alpha\) is non-\(0\) as an element of \(\ft\),
then
\(C_{\SO_{2n + 1}}(v_0)\) equals \(C_{\SO_{2n + 1}}(\ft^*)\).
Since the subgroup of \(W(\SO_{2n + 1}, T)(k)\) that
acts trivially on \(\ft^*\) is
generated by reflections in
the roots \(\alpha \in \Phi(\SO_{2n + 1}, T)\) such that
\(h_\alpha\) equals \(0\) in \(\ft\), we have that
$H:=C_{\SO_{2n+1}}(\ft^*)$ is connected.
Then \cite[Lemma 3.1]{KW76} gives that \(H\) is a reductive group..
Concretely,
\(H\) is isomorphic to $(\SO_3)^n$, and
generated by \(T\) and its short-root groups in \(\SO_{2n + 1}\); equivalently, it is the pre-image of \((\Sp_2)^n \subset \Sp_{2n}\) under
\(\SO_{2n + 1}\sra\Sp_{2n}\). The set
\[
\ft':=\{v\in\ft^*\;|\;C_G(v)=H\}
\]
is dense open.

Since
\(H\) is a connected, reductive group and
\(C_G(H)\) is contained in \(C_G(T) = T \subset H\), it follows that
the identity component of \(N_G(H)\) is
\((H\cdot C_G(H))^\circ = H\).
Thus, \(N_G(H)\) is smooth, and so
generated by $N_G(T)$ and $H$.  The Tits lifts (with respect to some fixed system of positive roots) of the simple long-root reflections generate a copy of \(S_n\) that is a direct complement to \(H\) in \(N_G(H)\).

A choice of one from each opposite pair of short roots gives a basis of \(\ft^*\), hence an
isomorphism $\ft^*\cong\mathbb{A}^n$ under which $N_G(H)/H\cong S_n$ acts by permuting the coordinates, and $\ft'$ is the locus with distinct non-zero coordinates.

In particular, the image $V:=\im(df^*:\sp_{2n}^*\ra\so_{2n+1}^*)$ of the dual of the differential of \(f\) is the sum of $\ft^*$ and those weight spaces of \(T\) on \(\fg^*\) corresponding to the long roots of \(T\) on $\so_{2n+1}$.
(In Section \ref{sec:second}, we identify
\(V\) with \(\so_{2n}^*\).)
Since $V$ contains \(\ft^*\) and is $\SO_{2n+1}$-stable, all semisimple elements of $\so_{2n+1}^*$ lie in $V$.

Consider the action map
\[
c:\SO_{2n+1}\times_{N_G(H)}\ft^*\ra V.
\]
This restricts to an isomorphism from $\SO_{2n+1}\times_{N_G(H)}\ft'$ onto its image, which is open in $V$ because the differential is non-singular (this is because $C_{\so_{2n+1}}(v)=\fh$ for every $v\in\ft'$). In particular $c$ is birational, and the image of $c$ is dense in $V$, proving Theorem \ref{thm:main}\ref{ss}. Consider the map
\[
\bar{c}:\ft^*\quo S_n\ra V\quo\SO_{2n+1}.
\]
We will prove that it is an isomorphism by constructing its inverse mimicking the proof in \cite[II.3.17\textquotesingle]{SS70}. Since both sides are affine varieties, we can construct the map at the level of ring homomorphisms, i.e., given an $S_n$-invariant regular function $f$ we shall construct an $\SO_{2n+1}$-invariant regular function on $V$. Indeed, $f$ defines a function on the domain of $c$, and thus a rational function on $V$, i.e., there are relatively prime polynomials $g,h\in\cO(V)$ such that $f\times c^*(h)=c^*(g)\in\cO(\SO_{2n+1}\times_{N_G(H)}\ft^*)$.
Since $f$ is by construction $\SO_{2n+1}$-invariant, so is $g/h$. Since $\SO_{2n+1}$ is semisimple and $\cO(V)$ is a polynomial ring, this means that both $g$ and $h$ are $\SO_{2n+1}$-invariant.
Since \(f\) is an isomorphism on \(T\), we have that
\(g\) vanishes on the zero set of \(h\) on \(\ft^*\), and hence
\(g\) vanishes on every semisimple element of
the zero set of \(h\).
Now fix an arbitrary \(v \in V\) in the zero set of \(h\).
Thanks to Lemma \ref{lem:orbit}, the closure of the
\(\SO_{2n + 1}\)-orbit of \(v\) contains
a semisimple element \(v\semi\), necessarily in \(V\), so that
\(g(v\semi)\) equals \(g(v)\) and
\(h(v\semi)\) equals \(h(v)\).
In particular, \(h(v\semi)\) equals \(h(v) = 0\), so $g(v)=g(v\semi)=h(v\semi)$ also equals $0$.
%so \(g(v\semi) = h(v\semi)\) equals \(0\), so \(g(v) = g(v\semi)\) also equals \(0\).

\iffalse
	\LS{I think that the previous sentences replace the old argument, which I reproduce below.}
Thanks to Lemma \ref{lem:orbit}, for
every orbit \(\cO\) of \(\SO_{2n + 1}\) on \(V\), the closure
	\LS{Changed `boundary' to `closure'.  Though the orbits are too small
to have interiors, I think that it's a distraction to
rely on that.}
of \(\cO\) contains a semisimple orbit $\cO_0$ \CC{changed from $\cO_{ss}$ to $\cO_0$ to agree with the rest}, i.e., one in the image of $c$,
	\LS{Isn't \(V\) closed anyway, so that of course the closure of
an orbit in \(V\) is again in \(V\)?}
and in particular $g$ has the same value on $\cO$ and $\cO_0$, and the same for $h$. For every $v\in V$ with $h(v)=0$, we put $\cO=\Ad^*(G)v$, $v_0\in\cO_0$ where $\cO_0$ is as above, and $h(v)=0\implies h(v_0)=0\implies g(v_0)=0\implies g(v)=0$.
	\LS{End of the old proof.}
\fi

By Hilbert's Nullstellensatz this implies that $h\mid g$, i.e., we may assume $h=1$. One verifies that $f\mapsto g$ is then our desired inverse, and therefore $\bar{c}$ is an isomorphism. This proves Theorem \ref{thm:main}\ref{iso2}.

That $\bar{c}$ is injective shows in particular that every two semisimple orbits go to different points in $V\quo\SO_{2n+1}$, and hence they have to be closed by Lemma \ref{lem:orbit}. This finishes the proof of Theorem \ref{thm:main}\ref{key}.

\section{Second approach}
\label{sec:second}

We outline another approach to the proof of
Theorem \ref{thm:main}\ref{key}, \ref{ss}, and \ref{iso2}.
We continue to assume that \(k\) has characteristic \(2\).

Let \(n\) be a positive integer, and
\(q\) a quadratic form on a (\(2n + 1\))-dimensional \(k\)-vector space \(V\)
such that,
if we write
\(b\) for the polar \((x, y) \mapsto q(x + y) - q(x) - q(y)\) of \(q\) and
\(\ell\) for the radical
\(\{x \in V \;|\; \text{\(b(x, y) = 0\) for all \(y \in V\)}\}\)
of \(b\), then
\(\ell\) is \(1\)-dimensional.
Then \(b\) gives a
non-degenerate, alternating form \(\bar b\) on
\(\bar V := V/\ell\),
every element of \(\SO(V, q)\) preserves \(\ell\), and
the resulting map \(\SO(V, q) \ra \GL(\bar V)\) is actually
the isogeny \(f : \SO(V, q) \ra \Sp(\bar V, \bar b)\) that
we used in the previous section.

Fix for the moment a complement \(V'\) to \(\ell\) in \(V\).
Extension trivially across \(\ell\) provides a map
\(\SO(V', q) \ra \SO(V, q)\),
which we regard as an inclusion.
We may also transport \(b\) to \(V'\), and then we have that
\(\SO(V', q)\) sits inside \(\Sp(V', b)\).
Then the relevant maps fit into a commutative diagram
\[\begin{tikzcd}
\SO(V, q) \ar[r, twoheadrightarrow] & \Sp(\bar V, \bar b) \ar[d, "\sim"{rotate=90, anchor=north}] \\
\SO(V', q) \ar[r, hookrightarrow]\ar[u, hookrightarrow] & \Sp(V', b).
\end{tikzcd}\]
	%% https://tex.stackexchange.com/questions/324957/how-to-draw-an-isomorphism-arrow-in-tikz-cd

We will soon fix a maximal torus in \(\SO(V', q)\), but
for the moment we want to be able to work with an arbitrary choice.

If \(T\) is a maximal torus in \(\SO(V', q)\), then
\(\SO(V', q)\) is generated by
\(T\) and its long-root subgroups in \(\SO(V, q)\);
the image of \(T\) is a maximal torus in \(\Sp(\bar V, \bar b)\), and
the image of \(\SO(V', q)\) in \(\Sp(\bar V, \bar b)\) is
the subgroup of the latter generated by
\(f(T)\) and
its short-root subgroups in \(\Sp(\bar V, \bar b)\); and
\(f\) restricts to an isomorphism of \(\SO(V', q)\) onto its image.
The subgroup \(\SO(V', q)\) of \(\SO(V, q)\), and
its isomorph in \(\Sp(\bar V, \bar b)\),
depend on the choice of \(V'\).
However, the image of the embedding of Lie algebras
\(\so(V', q) \ra \sp(\bar V, \bar b)\) does not depend on the choice of \(V'\).
In fact,
since the complement
\(\{X \in \so(V, q) \;|\; X V' \subset \ell\}\) to
(the image of) \(\so(V', q)\) in \(\so(V, q)\) is
annihilated by \(\so(V, q) \ra \sp(\bar V, \bar b)\), we have that
the images of \(\so(V', q)\) and \(\so(V, q)\) in \(\sp(\bar V, \bar b)\) are
equal.
This common image may be described
%in terms of \(T\) as
%the sum of \(\ft\) and its short-root spaces in \(\sp(\bar V, \bar b)\); and
without reference to \(V'\) as
\(\{\bar X \in \sp(\bar V, \bar b) \;|\;
	\text{\(\bar b(\bar X\bar v, \bar v) = 0\) for all
		\(\bar v \in \bar V\)}\}\).

We now write
\(\SO_{2n + 1}\) for \(\SO(V, q)\), and
\(\Sp_{2n}\) for \(\SO(\bar V, \bar b)\).
In order to emphasise that \(\so(V', q)\) does not
depend on the choice of \(V'\), we denote it by \(\so_{2n}\).
Thus we have a
sub-\(\Sp_{2n}\)-representation
\(\so_{2n}\) of \(\sp_{2n}\), and
dually a
quotient \(\Sp_{2n}\)-representation
\(\so_{2n}^*\) of \(\sp_{2n}^*\).
Of course, the quotient \(\SO_{2n + 1} \sra \Sp_{2n}\) allows us
to inflate any representation of \(\Sp_{2n}\) to
a representation of \(\SO_{2n + 1}\).
Then we have identified
the \(\SO_{2n + 1}\)-representation \(\so_{2n}^*\) with
\(\sp_{2n}^*/\ker(\sp_{2n}^* \ra \so_{2n + 1}^*)\), hence with
\(\im(\sp_{2n}^* \ra \so_{2n + 1}^*)\).
Since the composition
\(\so_{2n} \ra \so_{2n + 1} \ra \sp_{2n}\)
(which \textit{a priori} depends on the choice of \(V'\)) is the inclusion
\(\so_{2n} \ra \sp_{2n}\)
(which doesn't), also the composition
\(\sp_{2n}^* \ra \so_{2n + 1}^* \ra \so_{2n}^*\) is the quotient
\(\sp_{2n}^* \ra \so_{2n}^*\).
That is, \(\im(\sp_{2n}^* \ra \so_{2n + 1}^*)\) is
a canonical lift of \(\so_{2n}^*\) in \(\so_{2n + 1}^*\), which
we use to regard \(\so_{2n}^*\) as
a sub-\(\SO_{2n + 1}\)-representation of
\(\so_{2n + 1}^*\).

Now fix a maximal torus \(T\) in \(\SO(V', q)\).
Then \(T\) is also a maximal torus in \(\SO_{2n + 1}\), and
the isogeny \(\SO_{2n + 1} \ra \Sp_{2n}\) is an isomorphism of
\(T\) onto a maximal torus in \(\Sp_{2n}\).
We shall abuse notation by denoting the image torus also by \(T\).
Thus \(\ft\) is contained in \(\so_{2n}\).
Note that \(f\) restricts to an isomorphism
\(N_{\SO_{2n + 1}}(T)/T \ra N_{\Sp_{2n}}(T)/T\) that
intertwines their actions on \(T\).
Write \(W'\) for \(N_{\SO(V', q)}(T)/T\), and
\(W\) for \(N_{\SO_{2n + 1}}(T)/T\) and its isomorph \(N_{\Sp_{2n}}(T)/T\).
Then \(W'\) is a normal subgroup of \(W\), and
there is a complement to \(W'\) in \(W\) that
acts by sign changes on a basis for \(X_*(T)\), hence
trivially on \(\ft = X_*(T) \otimes_{\mathbb Z} k\) and on \(\ft^*\).
That is,
\(\ft\quo W'\) equals \(\ft\quo W\) and
\(\ft^*\quo W'\) equals \(\ft^*\quo W\).
	%% Really equal, not just isomorphic!
Recall that there are functions
\(a_0, \dotsc, a_n \in k[\so_{2n}]\) such that, for every
\(X \in \so_{2n}\), if
we view \(X\) as an element of \(\gl_{2n}\), then
the characteristic polynomial of \(X\) is the element
\((a_0(X) + a_1(X)\lambda + \dotsb + a_n(X)\lambda^n)^2\)
of \(k[\lambda]\).
Since \(k\) has characteristic \(2\),
these functions are uniquely determined.
In particular, they are not just \(\SO(V', q)\)-\nolinebreak, but actually
\(\Sp_{2n}\)-\nolinebreak, invariant.
Since \(a_0\rvert_\ft, \dotsc, a_n\rvert_\ft\) are
the elementary symmetric polynomials in the diagonal entries of
an element of \(\ft\),
they generate \(k[\ft]^{W'} = k[\ft]^W\).
Thus the restriction map
\(k[\so_{2n}]^{\SO(V', q)} \ra k[\ft]^{W'} = k[\ft]^W\) is
already surjective when restricted to
\(k[\so_{2n}]^{\Sp_{2n}}\); but
Theorem \ref{thm:main}\ref{iso2} gives that it is an isomorphism, so that
\(k[\so_{2n}]^{\Sp_{2n}}\) equals \(k[\so_{2n}]^{\SO(V', q)}\) and
the restriction map
\(k[\so_{2n}]^{\Sp_{2n}} \ra k[\ft]^W\) is
an isomorphism of \(k\)-algebras, i.e.,
the dual map \(\ft\quo W \ra \so_{2n}\quo\Sp_{2n}\) is
an isomorphism of schemes.
We will use this later to prove Theorem \ref{thm:main}\ref{iso2}.

Since
every semisimple element of \(\sp_{2n}\) belongs to
the Lie algebra of some maximal torus in \(\Sp_{2n}\)
\cite[Proposition 11.8]{Bo}; since
any two maximal tori in \(\Sp_{2n}\) are
\(\Sp_{2n}(k)\)-conjugate; and since
\(\so_{2n}\) is a sub-\(\Sp_{2n}\)-representation of \(\sp_{2n}\),
it follows that
\(\so_{2n}\) contains all semisimple elements of \(\sp_{2n}\).
That is, we have containments
\(\so_{2n}\semis \subset \sp_{2n}\semis \subset \so_{2n}\).
In particular, since
Theorem \ref{thm:main}\ref{ss}
(for \(\SO(V', q) \curvearrowright \so(V', q) = \so_{2n}\)) shows that
\(\overline{\so_{2n}\semis}\) equals \(\so_{2n}\), also
\(\overline{\sp_{2n}\semis}\) equals \(\so_{2n}\).
This shows Theorem \ref{thm:main}\ref{ss}
(for \(\Sp_{2n} \curvearrowright \sp_{2n}\)).
Therefore, the fact that
\(\ft\quo W \ra \so_{2n}\quo\Sp_{2n}\) is an isomorphism is exactly
Theorem \ref{thm:main}\ref{iso2} for
\(\Sp_{2n} \curvearrowright \sp_{2n}\).

Since the functoriality of the Jordan decomposition
\cite[Theorem 4.4(4)]{Bo} guarantees that
an element of \(\so_{2n}\) is semisimple if and only if
it is so as
an element of \(\sp_{2n}\), we have equality
\(\so_{2n}\semis = \sp_{2n}\semis\).
Now the \(k\)-points of
\(\so_{2n}\quo\SO(V', q) = \so_{2n}\quo\Sp_{2n}\) are,
on the one hand, exactly
the closed \(\SO(V', q)\)-orbits on \(\so_{2n}\), which are
the \(\SO(V', q)\)-orbits through elements of
\(\so_{2n}\semis = \sp_{2n}\semis\), and,
on the other hand, exactly
the closed \(\Sp_{2n}\)-orbits on \(\so_{2n}\).
This shows that \(\Sp_{2n}\)-orbits through
elements of \(\sp_{2n}\semis\) are closed.
Lemma \ref{lem:orbit} gives the converse, and so
establishes Theorem \ref{thm:main}\ref{key}.

Nearly the same argument works for
\(\SO_{2n + 1} \curvearrowright \so_{2n + 1}^*\),
but there are two major problems.
First, we do not know an easy, concrete description of
\(k[\so_{2n}^*]^{\SO(V', q)}\) that allows us to see that it equals
\(k[\so_{2n}^*]^{\SO_{2n + 1}}\).
Second, although
each semisimple element of \(\so_{2n + 1}^*\) belongs to
\(\ft_1^*\) for some maximal torus \(T_1\) in \(\SO_{2n + 1}\), and to
\(\so_{2n}^*\), there is no guarantee \textit{a priori} that
we can take the maximal torus \(T_1\) in \(\SO(V', q)\).
That is, although we have the inclusion
\((\so_{2n + 1}^*)\semis \subset \so_{2n}^*\),
it is no longer clear that the inclusion
\((\so_{2n}^*)\semis \subset (\so_{2n + 1}^*)\semis\)
is an equality.
An easy way to deal with these problems is to move
back from the dual to the primal side.

We do this by exhibiting an \(\SO_{2n + 1}\)-invariant,
non-degenerate, bilinear form on \(\so_{2n}\).
We start with the polar \(B_{\mathbb Z}\) of
the quadratic form on \(\sp_{2n, \mathbb Z}\)
defined in \cite[\S I.3]{Kot16}, and denoted there by \(Q_2\).
Concretely, it is not hard to show, though we do not need to use, that
\(B_{\mathbb Z}(X, Y)\) equals
\(\Tr(X Y)\) for all \(X, Y \in \sp_{2n, \mathbb Z}\).
Let us momentarily write \(\so_{2n, \mathbb Z}\) for
the subspace of \(\sp_{2n, \mathbb Z}\) spanned by
\(\ft_{\mathbb Z} := X_*(T)\)
and
the short-root subspaces for \(T_{\mathbb Z}\).
This is an abuse of notation, because
\(\so_{2n, \mathbb Z}\) is not
a subalgebra of \(\sp_{2n, \mathbb Z}\), but
\(\so_{2n, \mathbb Z} + 2\sp_{2n, \mathbb Z}\) is a
\(\Sp_{2n, \mathbb Z}\)-stable subalgebra on which
\cite[Lemma I.2(1) and \S I.5]{Kot16} gives that
\(B_{\mathbb Z}\) is \(2\Z\)-valued, so we may write
\(B_{\mathbb Z}'\) for \(1/2\) times the restriction of \(B\).
Thus, \(B_{\mathbb Z}'\) is a \(\mathbb Z\)-valued form.
Then we write \(B'\) for the \(k\)-valued form on
\(\so_{2n} =
\left((\so_{2n, \mathbb Z} + 2\sp_{2n, \mathbb Z})/2\sp_{2n, \mathbb Z}\right)\otimes_{\mathbb Z} k\)
induced by \(B_{\mathbb Z}'\).
Concretely, for any \(X, Y \in \so_{2n}\), we can choose lifts
\(\til X\) and \(\til Y\) of \(X\) and \(Y\) in
\(\so_{2n, \mathbb Z} \otimes_{\mathbb Z} W(k)\).
Then \(\Tr(\til X\til Y)\) belongs to \(2 W(k)\), and the image of
\(\frac1 2\Tr(\til X\til Y)\) in \(k\) is
\(B'(X, Y)\).
The pairing \(B'\) is
\(\Sp_{2n}\)-, hence \(\SO_{2n + 1}\)-, invariant because
the original form \(B\) was \(\Sp_{2n, \mathbb Z}\)-invariant
(or because one can verify it directly from
the concrete description).
It is non-degenerate on \(\ft\) because
it is Weyl invariant, so that
the basis of \(\ft\) consisting of coroots to
long roots in \(\Phi(\Sp_{2n}, T)\) is orthogonal; and
\cite[Lemma I.1]{Kot16} says that, if
\(\alpha \in \Phi(\Sp_{2n}, T)\) is long, then
\(B'(h_\alpha, h_\alpha)\) equals \(1\).
and then another application of
\cite[\S I.5]{Kot16}, together with the explicit description of
\cite[\S I.4]{Kot16}, gives that it is also non-degenerate on
the orthogonal complement to \(\ft\) in \(\so_{2n}\).

We now have a duality isomorphism
\(\so_{2n} \to \so_{2n}^*\) of
representations of \(\SO_{2n + 1}\), hence
an isomorphism of
\(k[\so_{2n}]^{\Sp_{2n}} = k[\so_{2n}]^{\SO_{2n + 1}}\) with
\(k[\so_{2n}^*]^{\SO_{2n + 1}}\) that restricts to
an isomorphism of
\(k[\so_{2n}]^{\SO(V', q)}\) with
\(k[\so_{2n}^*]^{\SO(V', q)}\).
Since we have already shown that
\(k[\so_{2n}]^{\SO(V', q)}\) equals \(k[\so_{2n}]^{\Sp_{2n}}\),
it follows that
\(k[\so_{2n}^*]^{\SO(V', q)}\) equals \(k[\so_{2n}^*]^{\SO_{2n + 1}}\).
That is, \(\so_{2n}^*\quo\SO(V', q)\) equals
\(\so_{2n}^*\quo\SO_{2n + 1}\).
Thus, as for
\(\Sp_{2n} \curvearrowright \sp_{2n}\), we have that
Theorem \ref{thm:main}\ref{key}, \ref{iso2} for
\(\SO_{2n + 1} \curvearrowright \so_{2n + 1}^*\) will follow from
the analogous results for
\(\SO(V', q) \curvearrowright \so_{2n}^*\) once we show
Theorem \ref{thm:main}\ref{ss}, i.e., that
\((\so_{2n + 1}^*)\semis\) equals \((\so_{2n}^*)\semis\).

First, remember that
we have identified \(\ft^*\) with
the \(T\)-fixed subspaces of
\(\so_{2n}^*\) and \(\so_{2n + 1}^*\); and then
the quotient \(\so_{2n + 1}^* \ra \so_{2n}^*\)
restricts to the identity on \(\ft^*\).
In this sense, we may say that
\(\so_{2n}^*\), as a sub-\(\SO_{2n + 1}\)-representation of
\(\so_{2n + 1}^*\), contains \(\ft^*\).
It is clear that
\((\so_{2n}^*)\semis := \Ad^*(\SO(V', q))\ft^*\) is
contained in
\((\so_{2n + 1}^*)\semis = \Ad^*(\SO_{2n + 1})\ft^*\).
For the reverse containment,
as observed in Remark \ref{rem:Killing}, our
\(\SO_{2n + 1}\)-equivariant isomorphism of
\(\so_{2n}\) with \(\so_{2n}^*\) restricts to an
\(\SO_{2n + 1}\)-equivariant bijection
\(\sp_{2n}\semis = \so_{2n}\semis \ra (\so_{2n}^*)\semis\).
Since \(\sp_{2n}\semis\) is preserved by \(\Sp_{2n}\), and hence by
\(\SO_{2n + 1}\), so too is \((\so_{2n}^*)\semis\).
This establishes the reverse containment of
\(\Ad^*(\SO_{2n + 1})\ft^* = (\so_{2n + 1}^*)\semis\) in
\((\so_{2n}^*)\semis\), so gives equality.

\section{Non-uniqueness of the Jordan decomposition}\label{sec:nonuni}

Suppose that there are
roots \(\alpha, \beta \in \Phi(G, T)\)
such that \(\alpha + \beta\) belongs to \(\Phi(G, T)\), but
the coroot \(h_{\alpha + \beta}\) does not
lie in the span of the coroots \(h_\alpha\) and \(h_\beta\).

\begin{remark}
\label{rem:nonclosed}
The irreducible component of \(\Phi(G, T)\) containing
\(\alpha\) and \(\beta\) is not
simply laced,
at least one of \(\alpha\) or \(\beta\) is short,
\(\alpha + \beta\) is long, and
the length of \(\alpha + \beta\) is \(0\) in \(k\).
%% Since \(r \mapsto \ell(r)r^\vee\) extends to
%% an additive map from the root lattice to the coroot lattice.
Since the closed subsystem
\(\Phi_{\alpha, \beta} :=
	\Phi(G, T) \cap (\mathbb Z\alpha + \mathbb Z\beta)\) of
\(\Phi(G, T)\) generated by \(\alpha\) and \(\beta\) has
rank \(2\) and is not simply laced, we have that
the characteristic of \(k\) equals \(2\) and
\(\Phi_{\alpha, \beta}\) is of type \(\mathsf B_2 = \mathsf C_2\); or
the characteristic of \(k\) equals \(3\) and
\(\Phi_{\alpha, \beta}\) is of type \(\mathsf G_2\).
We now proceed by exhaustion, first choosing
a system of simple roots for \(\Phi\) containing \(\alpha\).
Then running through the possibilities shows that,
having fixed \(\alpha\),
there are two possibilities for \(\beta\), which are conjugate in
\(\operatorname{stab}_{W(\Phi_{\alpha, \beta})}(\alpha)\); and
\(\alpha\) and \(\beta\) are both short.
Thus, if we write \(H\) for the subgroup of \(G\) generated by \(T\) and
the root subgroups for \(T\) corresponding to
roots in \(\Phi_{\alpha, \beta}\), then
there is a subgroup \(H'\) of \(H\) containing \(T\) such that
\(\Phi(H', T)\) consists precisely of
the short roots in \(\Phi(H, T) = \Phi_{\alpha, \beta}\).
Specifically, we are in one of two situations.
\begin{enumerate}
\item
The characteristic of \(k\) is \(2\),
the derived subgroup of \(H\) is \(\SO_5\) or \(\Sp_4\),
the derived subgroup of \(H'\) is \(\operatorname{PGO}_4\) or \(\SO_4\), and
\(\alpha\) and \(\beta\) are orthogonal short roots.
%the \(\alpha\)-string through \(\beta\) is
%\(-\alpha + \beta, \beta, \alpha + \beta\).
Then \(\Phi_{\alpha, \beta}\) is the union of the set
\(\pm\{\alpha, \beta\}\) of short roots and the set
\(\pm\{\alpha + \beta, \alpha - \beta\}\) of long roots.
\item
The characteristic of \(k\) is \(3\),
the derived subgroup of \(H\) is \(G_2\),
the derived subgroup of \(H'\) is \(\operatorname{PGL}_3\), and
\(\alpha\) and \(\beta\) are short roots
%that make an acute angle.
whose sum is a long root.
%the \(\alpha\)-string through \(\beta\) is
%\(-2\alpha + \beta, -\alpha + \beta, \beta, \alpha + \beta\).
Then \(\Phi_{\alpha, \beta}\) is the union of the set
\(\pm\{\alpha, \beta, \alpha - \beta\}\) of short roots and the set
\(\pm\{\alpha + \beta, 2\alpha - \beta, \alpha - 2\beta\}\) of long roots.
\end{enumerate}
%In particular, since the \(\alpha\)- and \(\beta\)-root groups for \(T\)
%lie in \(H'\), but \(\alpha + \beta\) does not lie in \(\Phi(H', T)\), then
%the \(\alpha\)- and \(\beta\)-root groups for \(T\) commute.
\end{remark}

Remark \ref{rem:nonclosed} gives that
\(U_\alpha\) and \(U_{\alpha + \beta}\) commute.
Further, for each non-trivial
\(X_{-\alpha - \beta} \in \fg_{-\alpha - \beta}\) and each
\(k\)-algebra \(A\), the map
\(U_\alpha(A) \to \fg_{-\beta} \otimes_k A\) given by
\(u_\alpha \mapsto (\Ad(u_\alpha) - 1)X_{-\alpha - \beta}\) is an
%\(T(A)\)-equivariant
isomorphism.

Now, since \(h_{\alpha + \beta}\) does not belong to
the span of \(h_\alpha\) and \(h_\beta\), there is some
\(X^*\semi \in \ft^*\) that vanishes on \(h_\alpha\) and \(h_\beta\), but
satisfies \(X^*\semi(h_{\alpha + \beta}) = 1\).
In particular, \(X^*\semi\) is fixed by \(U_\alpha\)
\cite[Lemma 3.1ii)]{KW76}.
Let \(X^*\nilp\) be a non-zero element of the
\(\beta\)-weight space for \(T\) in \(\fg\), so that
\(X^*\nilp\) vanishes on the \(T\)-stable complement to
\(\fg_{-\beta}\) in \(\fg\).
Since \(-2\alpha - \beta\) does not belong to \(\Phi(G, T)\),
we have that
\(U_{\alpha + \beta}\) fixes \(X^*\nilp\).
Then \(X^* := X^*\semi + X^*\nilp\) has an obvious Jordan decomposition.

Now the Chevalley commutation relations say that the map
\(U_{\alpha + \beta}(k) \ra (\fg_{-\alpha - \beta})^*\) given by
\(u_{\alpha + \beta} \mapsto (\Ad^*(u_{\alpha + \beta}) - 1)X^*\semi\)
is a \(T(k)\)-equivariant isomorphism, in particular
linear for
the \(T\)-invariant linear structure on \(U_{\alpha + \beta}\);
and, by our description of \(\Phi_{\alpha, \beta}\) above, that the map
\(U_\alpha(k) \ra (\fg_{-\alpha - \beta})^*\) given by
\(u_\alpha \mapsto (1 - \Ad^*(u_\alpha))X^*\nilp\) is also an isomorphism,
this time not \(T(k)\)-equivariant, but still
linear for
the \(T\)-invariant linear structure on \(U_\alpha\).
We thus obtain an isomorphism of vector groups
\(f : U_{\alpha + \beta} \ra U_\alpha\).

Remember that \(U_\alpha\) and \(U_{\alpha + \beta}\) commute.
In particular, the map
\(U_{\alpha + \beta} \ra G\) given by
\(u_{\alpha + \beta} \mapsto f(u_{\alpha + \beta})u_{\alpha + \beta}\) is
a homomorphism.
Let \(A\) be a \(k\)-algebra.
For every element
\(u_{\alpha + \beta} \in U_{\alpha + \beta}(A)\), if
we put \(u_\alpha = f(u_{\alpha + \beta})\), then,
since
\(U_\alpha\) fixes \(X^*\semi\) and
\(U_{\alpha + \beta}\) fixes \(X^*\nilp\), we have the equalities
\[
(1 - \Ad^*(u_\alpha u_{\alpha + \beta}))X^*\nilp =
(1 - \Ad^*(u_\alpha))X^*\nilp =
(\Ad^*(u_{\alpha + \beta}) - 1)X^*\semi =
(\Ad^*(u_\alpha u_{\alpha + \beta}) - 1)X^*\semi.
\]
That is, our homomorphism
\(U_{\alpha + \beta} \ra G\) actually lands in
\(C_G(X^*)\), hence in \(C_G(X^*)_\text{red}^\circ\).
Moreover, if \(A\) equals \(k\) and
\(u_{\alpha + \beta}\) is non-trivial, then
\(\Ad^*(u_\alpha u_{\alpha + \beta})X^*\semi =
\Ad^*(u_{\alpha + \beta})X^*\semi\) is a semisimple part of
\(X^*\) that is
different from \(X^*\semi\).
Note that this shows that each of the following claims can fail:
\begin{itemize}
\item the claim in
\cite[\S3.8i)]{KW76} that the map
\(B/C_B(X^*\semi) \ra \mathbb A^m\) defined there
(given by evaluating at certain root vectors)
is an isomorphism;
\item the claim in
\cite[\S3.8ii)]{KW76}, that we have
the containment written there as
\(C_G(X^*)^\circ \subset C_G(X^*\semi)^\circ\)
(but actually meaning, given
the implicit passage to maximal reduced subschemes,
\(C_G(X^*)_\text{red}^\circ \subset C_G(X^*\semi)^\circ\)); and
\item the claim in
\cite[Theorem 4iv)]{KW76}, that
Jordan decompositions are unique.
\end{itemize}

So far, our computations have been abstract.
Let us now give a concrete example.
Let $k$ be an algebraically closed field of characteristic $2$, and
$G=\Sp_4$ the closed subgroup of $\GL_4$ stabilizing the symplectic form:
\[
\left\langle\matr{x_1\\x_2\\x_3\\x_4},\matr{y_1\\y_2\\y_3\\y_4}\right\rangle=x_1y_4+x_2y_3+x_3y_2+x_4y_1.
\]
The subgroup \(T\) of diagonal matrices in \(G\) is a maximal torus, and,
with hopefully obvious notation for the characters
\(\epsilon_1, \dotsc, \epsilon_4\) of
the group of diagonal matrices in \(\GL_4\), so that
the restrictions of
\(\epsilon_3\) and \(\epsilon_4\) to \(T\) agree with
the restrictions of
\(-\epsilon_2\) and \(-\epsilon_1\), we will work here with
\(\alpha = \epsilon_1 - \epsilon_2\) and
\(\beta = \epsilon_1 + \epsilon_2\).

The embedding of \(G\) in \(\GL_4\) gives an embedding of \(\fg\) in \(\gl_4\), hence a quotient \(\gl_4^* \ra \fg^*\).
We further identify $\gl_4^*$ with $\gl_4$ via the trace pairing. We will then represent elements in $\fsp_4^*$ non-uniquely by $4\times4$ matrices.
We abuse notation by using equality to denote a choice of representative,
which allows us to write, for example,
\[
X^* =
\matr{1&&1&\\&0&&\\&&1&\\&&&0} =
\matr{1&&1&\\&1&&\\&&0&\\&&&0},
\]
where all blank entries are filled by $0$,
for an element of \(\fsp_4^*\).
It has an evident Jordan decomposition:
\begin{equation}
\label{eq:Xss}
X^*=X^*\semi+X^*\nilp,\;\;\text{where}\;\;X^*\semi=\matr{1&&&\\&0&&\\&&1&\\&&&0},\;\;X^*\nilp=\matr{0&&1&\\&0&&\\&&0&\\&&&0}.
\end{equation}
The morphism \(U_{\alpha + \beta} \ra U_\alpha\) is given by
\[
\matr{1&&&t\\&1&&\\&&1&\\&&&1} \mapsto
\matr{1&t&&\\&1&&\\&&1&t\\&&&1}.
\]
Specifically, consider the elements $u_\alpha \in U_\alpha(k)$ and $u_{\alpha + \beta} \in U_{\alpha + \beta}(k)$ given by
\[
u_\alpha=\matr{1&1&&\\&1&&\\&&1&1\\&&&1},\;\;u_{\alpha + \beta}=\matr{1&&&1\\&1&&\\&&1&\\&&&1}.
\]
They satisfy $u_\alpha u_{\alpha + \beta}=u_{\alpha + \beta}u_\alpha$ and $u_\alpha^2=u_{\alpha + \beta}^2=\Id$. Since $\fsp_4\ira\gl_4$ and thus $\gl_4^*\sra\fsp_4^*$ is $\Sp_4$-equivariant, we compute that
\[
\Ad^*(u_\alpha)X^*\semi=\matr{1&1&&\\&0&&\\&&1&1\\&&&0}=\matr{1&0&&\\&0&&\\&&1&0\\&&&0}=X^*\semi.
\]
Also it is easy to see that $\Ad^*(u_{\alpha + \beta})X^*\nilp=X^*\nilp$. On the other hand we have
\[
\Ad^*(u_{\alpha + \beta})X^*\semi=\matr{1&&&1\\&0&&\\&&1&\\&&&0}\ne X^*\semi,\;\;\Ad^*(u_\alpha)X^*\nilp=\matr{0&&1&1\\&0&&\\&&0&\\&&&0}\ne X^*\nilp.
\]
In particular $\Ad^*(u_\alpha u_{\alpha + \beta})X^*=X^*$, but $\Ad^*(u_\alpha u_{\alpha + \beta})X^*\semi\ne X^*\semi$.
%Hence $X^*=\Ad^*(u_\alpha u_{\alpha + \beta})X^*\semi+\Ad^*(u_\alpha u_{\alpha + \beta})X^*\nilp$ is another Jordan decomposition. In particular \cite[\S3.8ii) and Theorem 4iv)]{KW76} are incorrect.

\begin{remark}
\label{rem:no-restricted-dual}
One might wonder if there is a sort of
``restricted dual Lie algebra'' structure, i.e., a
\(p\)th-power map analogous to the one carried by a
restricted Lie algebra, on \(\fg^*\).
For every maximal torus \(T\),
there is such an operation on \(\ft^*\), where, for
every \(Y^* \in \ft^*\), we remember that
the \(p\)th-power map \((\cdot)^{[p]}\) on \(\ft\) is bijective, say
with inverse \((\cdot)^{-[p]}\), and then define
\((Y^*)^{[p]}(h) = \bigl(Y^*(h^{-[p]})\bigr)^p\) for
every \(h \in \ft\).
This structure is obviously functorial.
Moreover, for every maximal torus \(T\) in \(G\) and every
\(\alpha \in \Phi(G, T)\), we have that
\(h_\alpha^{[p]}\) equals \(h_\alpha\);
	%% For example, because \(h_\alpha\) is defined over \(\Z\), and
	%% this equality holds on all \(\Z\)-points.
so, for every \(Y^* \in \ft^*\), we have that
\(\{\alpha \in \Phi(G, T) \;|\;
	Y^*(h_\alpha) = 0\}\) equals
\(\{\alpha \in \Phi(G, T) \;|\;
	(Y^*)^{[p]}(h_\alpha) = 0\}\), and hence by
\cite[Lemma 3.1]{KW76} that
\(C_G(Y^*)^\circ\) equals \(C_G((Y^*)^{[p]})^\circ\).

If \(Y^* \in \fg^*\) is semisimple and
\(T_1\) and \(T_2\) are maximal tori in \(C_G(Y^*)^\circ\), and
if we write
\((Y^*)^{[p]}_1 \in \ft_1^* \subset \fg^*\) and
\((Y^*)^{[p]}_2 \in \ft_2^* \subset \fg^*\) for
the resulting \(p\)th powers, then
there is some
\(g \in C_G(Y^*)^\circ(k) = C_G((Y^*)^{[p]}_1)(k)\) such that
\(g T_1 g^{-1}\) equals \(T_2\), and
\(g\) both fixes \((Y^*)^{[p]}_1\) and carries it to \((Y^*)^{[p]}_2\).
That is,
\((Y^*)^{[p]}_1\) equals \((Y^*)^{[p]}_2\), so
the \(p\)th power \((Y^*)^{[p]}\) is
independent of the choice of
maximal torus in \(C_G(Y^*)^\circ\).

We may thus ask if there is some \(p\)th-power operation on
all of \(\fg^*\) that satisfies the following properties,
analogous to those for restricted Lie algebras:
\begin{itemize}
\item it is equivariant for the coadjoint action of \(G(k)\);
\item it restricts to the operation that
we have just defined on \((\fg^*)\semis\);
\item it respects Jordan decompositions; and
\item every nilpotent element of \(\fg^*\) is
annihilated by some iterate of the operation.
\end{itemize}
Unfortunately, there is not.
If there were, then, for the element \(X^* \in \fg^*\) above, we would
have for sufficiently large \(N\) that
\((X^*)^{[p]^N}\) was equal to both
\((X^*\semi)^{[p]^N}\) and
\(\Ad^*(u_\alpha u_{\alpha + \beta})(X^*\semi)^{[p]^N}\), and hence that
\(C_G(X^*\semi)^\circ = C^G((X^*\semi)^{[p]^N})^\circ\) was equal to
\(C_G(\Ad^*(u_\alpha u_{\alpha + \beta})(X^*\semi)^{[p]^N})^\circ =
C_G(\Ad^*(u_\alpha u_{\alpha + \beta})X^*\semi)^\circ =
(u_\alpha u_{\alpha + \beta})C_G(X^*\semi)^\circ(u_\alpha u_{\alpha + \beta})^{-1}\).
It would follow finally that
\(C_G(X^*\semi)^\circ = u_\alpha^{-1} C_G(X^*\semi)^\circ  u_\alpha\)
was equal to
\(u_{\alpha + \beta}C_G(X^*\semi)^\circ u_{\alpha + \beta}^{-1}\), but
this is false, since
\(u_{\alpha + \beta}T u_{\alpha + \beta}^{-1}\) is
not contained in \(C_G(X^*\semi)^\circ\)
(because the \(T\)-equivariant projection of
\(\Lie(u_{\alpha + \beta}T u_{\alpha + \beta}^{-1})\) on
\(\fg_{\alpha + \beta}\) is
non-trivial).
\end{remark}

\section{Uniqueness of the Jordan decomposition}
\label{sec:uni}

Throughout this section, let
\(X^*\) be an element of \(\fg^*\).

We begin with a crude observation,
Remark \ref{rem:crude-uni}, on the uniqueness of
a Jordan decomposition whose semisimple part belongs to
the dual Lie algebra of a fixed maximal torus.

\begin{remark}
\label{rem:crude-uni}
Suppose that
\(T\) is a maximal torus, and
\(X^*_1\) and \(X^*_2\) are both semisimple parts of \(X^*\) that
belong to \(T\).
Then \(X^*_1\) and \(X^*_2\) both
agree with \(X^*\) on \(\ft\) and
vanish on the unique \(T\)-stable complement to \(\ft\) in \(\fg\), so
they are equal.
\end{remark}

Lemma \ref{lem:CX-vs-CXs} is used in the proof of
our uniqueness result Theorem \ref{uni} to
provide a handy `reference' Jordan decomposition.

\begin{lemma}
\label{lem:CX-vs-CXs}
If \(T\) is a diagonalizable subgroup of \(C_G(X^*)\), then
there is a Jordan decomposition of \(X^*\) whose
semisimple part belongs to \((\fg^*)^T\).
\end{lemma}

\begin{proof}
We have that \(X^*\) belongs to \((\fg^*)^T\), which is
the subspace of \(\fg^*\) that
vanishes on the unique \(T\)-stable complement to
\(\Lie(C_G(T)^\circ)\) in \(\fg\), and is
\(C_G(T)^\circ\)-equivariantly identified
by restriction to \(\Lie(C_G(T)^\circ)\) with
\(\Lie(C_G(T)^\circ)^*\).
Remark \ref{rem:from-full-rank} shows that the pull-back to
\((\fg^*)^T\) of
a Jordan decomposition in \(\Lie(C_G(T)^\circ)^*\) is
a Jordan decomposition in \(\fg^*\).
Then the result follows from
Theorem \ref{thm:main}\ref{Jordan},
with \(C_G(T)^\circ\) in place of \(G\).
\end{proof}

Let \(X^*\semi\) be the semisimple part of
a Jordan decomposition of \(X^*\).
By Remark \ref{rem:dual-Jordan}, there are
a maximal torus \(T\) in \(C_G(X^*\semi)\) and
a Borel subgroup \(B\) of \(G\) containing \(T\) such that
\(X^*\) agrees on \(\fb\) with \(X^*\semi\), and
\(X^*\) vanishes on the \(T\)-stable complement to
\(\Lie(C_G(X^*\semi)^\circ)\) in \(\fg\).
By \cite[Lemma 3.1]{KW76}, we have that
\(\Phi(C_G(X^*\semi)^\circ, T)\) is precisely the set of
\(\alpha \in \Phi(G, T)\) such that
\(X^*\semi(h_\alpha)\), or, equivalently, \(X^*(h_\alpha)\),
equals \(0\).

Write \(Z\) for the schematic centre of \(C_G(X^*\semi)^\circ\),
%and \(A\) for the maximal torus in \(Z\),
and
put \(H = C_G(Z)^\circ\).
%and \(M = C_G(A)\).
We see here the main difference, for our purposes, between
the Lie algebra and the dual Lie algebra.
Namely, it is possible that \(H\) is strictly larger than
\(C_G(X^*\semi)^\circ\);
but the analogous construction on the Lie-algebra side always gives
\(C_G(Z(C_G(X\semi)^\circ))^\circ = C_G(X\semi)^\circ\) for
\(X\semi\) a semisimple element of \(\fg\).

By \cite[Remark 3.2]{Spi21}, there is an \(H\)-stable complement
\(\fh^\perp\) to \(\fh\) in \(\fg\).
It is the sum of the non-trivial-weight spaces for \(Z\) on \(\fg\).
Since \(\fh^\perp\) contains the
\(T\)-stable complement to \(\Lie(C_G(X^*\semi)^\circ)\) in \(\fg\),
we have that
\(X^*\) is trivial on \(\fh^\perp\), hence
fixed by \(Z\).

In characteristic \(0\), the centralizer of a
semisimple element is a Levi subgroup, so, unless
the Levi subgroup is all of \(G\), it has
a non-trivial central torus.
This need no longer happen in positive characteristic.
There are two failure modes.
Suppose that the characteristic of \(k\) is \(2\), and
\(G\) equals \(\Sp_4\).
Then the element \(X^*\semi\) given in \eqref{eq:Xss} has
the property that
\(C_G(X^*\semi)^\circ\) equals \(\SO_4\), whose centre is
\(Z(G)\).
This situation can only happen
(for an element \(X^*\semi\) that is not fixed by all of \(G\))
if the root system of \(C_G(X^*\semi)^\circ\) is not
closed in the root system of \(G\).
It is thus only visible in characteristic \(2\) and \(3\)
(see Remark \ref{rem:nonclosed}), and
is addressed by passing from \(C_G(X^*\semi)\) to \(H\).
Another failure mode occurs if
\(X^*\semi\) is given by
\(\begin{pmatrix}
a & b & c & d \\
e & f & g & c \\
i & j & -f & -b \\
n & i & -e & -a
\end{pmatrix} \mapsto a\).
Then \(C_G(X^*\semi)^\circ\) and \(H\) both
equal \(\Sp_2 \times \Sp_2\), whose centre is finite.
However, the centre in this case is at least a
non-trivial, connected group.
This phenomenon can occur in characteristic \(5\), as well as
\(2\) and \(3\).
Lemma \ref{lem:some-centre} shows that this is
the worst that can happen.

\begin{lemma}
\label{lem:some-centre}
If \(H\) does not equal \(G\), then
\((Z(H)/Z(G))^\circ\) is non-trivial.
\end{lemma}

\begin{proof}
We prove the result by induction on the dimension of \(G\).
It is vacuously true if the dimension is \(0\).
Suppose that we have proven the result for
all groups of smaller dimension than \(G\).

If the semisimple rank \(\dim(T) - \dim(Z(H))\) of \(H\) is
smaller than the semisimple rank \(\dim(T) - \dim(Z(G))\) of \(G\), then
\(\dim(Z(H)/Z(G)) = \dim(Z(H)) - \dim(Z(G))\) is positive, so
the result follows.
Otherwise, by \cite[I.4.5]{SS70}, there are a
system \(\Delta\) of simple roots for \(T\) in \(G\)
(not necessarily the one coming from \(B\)) and
an element \(\beta \in \Delta\) such that
the coefficient \(p\) of \(\beta\) in
the highest root \(\tilde\alpha\) in its component of \(\Phi(G, T)\) is
prime, and
\(\Phi(H, T)\) is contained in
\(\Z(\{\tilde\alpha\} \cup \Delta \setminus \{\beta\})\).
The coefficient \(p\) might or might not be
the characteristic of \(k\).
Write \(\varpi^\vee\) for the fundamental coweight of
\(T/Z(G)\) corresponding to \(\beta\).
Note that \(H\) is contained in the group
\(C_G(\varpi^\vee(\mu_p))^\circ\), which
has smaller dimension than \(G\).
If the containment is not an equality, then
we may replace \(G\) by \(C_G(\varpi^\vee(\mu_p))^\circ\), and
reason by induction.
Thus we may, and do, assume that we have equality.
In particular,
\(\Z\Phi(H, T)\) equals
\(\Z(\{\tilde\alpha\} \cup \Delta \setminus \{\beta\})\); so
\(X^*(h_\beta)\) is non-zero.

Write \(\ell\) for the normalized squared-length function on
\(\Phi(G, T)\), so that
\(\ell\) takes values in \(\{1, 2, 3\}\), and
\(\ell(\alpha)\) equals \(1\) whenever
\(\alpha\) belongs to a simply laced component of \(\Phi(G, T)\) or
is short in its component.
Then the map \(\Phi(G, T) \ra \Z\Phi^\vee(G, T)\) given by
\(\alpha \mapsto \ell(\alpha)\alpha^\vee\) extends to
a homomorphism \(\Z\Phi(G, T) \ra \Z\Phi^\vee(G, T)\).
We have that \(p\beta\) belongs to
\(\Z\Phi(H, T) = \Z\Phi(C_G(X^*\semi)^\circ, T)\), so that
\(p\ell(\beta)\beta^\vee\) is a linear combination of the various
\(\ell(\alpha)\alpha^\vee\) with
\(\alpha \in \Phi(C_G(X^*\semi)^\circ, T)\).
Since \(X^*(h_\alpha)\) equals zero for every
\(\alpha \in \Phi(C_G(X^*\semi)^\circ, T)\), but
\(X^*(h_\beta)\) is non-zero
(because \(\beta\) is not in \(\Phi(H, T)\), hence
not in \(\Phi(C_G(X^*\semi)^\circ, T)\)), we have that
the image in \(k\) of \(p\) or \(\ell(\beta)\) equals zero.

If the image of \(p\) in \(k\) equals zero, then
\(\varpi^\vee(\mu_p)\) is a non-trivial, connected subgroup of
\(Z(H)/Z(G)\).

Thus we may, and do, suppose that the image of \(p\) is non-zero, so
the image of \(\ell(\beta)\) equals zero.
By inspection of the Dynkin diagrams, this means that
the characteristic of \(k\) is \(2\),
the component of \(\Phi(G, T)\) containing \(\beta\) is
of type \(\mathsf F_4\), and, in the
Bourbaki numbering \cite[Planche VIII]{BouLALG},
\(\beta\) equals \(\alpha_2\).
Then each component of \(\Phi(H, T)\) is simply laced
(in fact, of type \(\mathsf A_2\)), so
\(\Phi(C_G(X^*\semi)^\circ, T) = \Phi(C_H(X^*\semi)^\circ, T)\) is
closed in \(\Phi(H, T)\).
Since also \(\Z\Phi(C_G(X^*\semi)^\circ, T)\) equals
\(\Z\Phi(H, T)\), we have that in fact
\(\Phi(C_G(X^*\semi)^\circ, T)\) equals \(\Phi(H, T)\), so
\(C_G(X^*\semi)^\circ\) equals \(H\).
Now, continuing with the Bourbaki numbering, we have that
\(3\beta = 3\alpha_2\) equals
\(\tilde\alpha - 2\alpha_1 - 4\alpha_3 - 2\alpha_4\).
Still using \cite[Planche VIII]{BouLALG}, we have that
\(\alpha_1\), \(\alpha_2\), and \(\tilde\alpha\) are long, while
\(\alpha_3\) and \(\alpha_4\) are short, so it follows that
\(6\beta^\vee\) equals
\(2\tilde\alpha^\vee - 4\alpha_1^\vee - 4\alpha_3^\vee - 2\alpha_4^\vee\), so
\(3\beta^\vee\) equals
\(\tilde\alpha^\vee - 2\alpha_1^\vee - 2\alpha_3^\vee - \alpha_4^\vee\).
In particular,
\(X^*\semi(h_\beta) = 3X^*(h_\beta)\) lies in the span of
\(\{X^*\semi(h_\alpha) \;|\; \alpha \in \Phi(H, T)\}\), which, because
\(H\) equals \(C_G(X^*\semi)^\circ\), is the singleton \(0\).
Thus \(X^*\semi(h_\beta)\) equals \(0\), so
\(\beta\) belongs to \(\Phi(C_G(X^*\semi)^\circ, T) = \Phi(H, T)\).
This is a contradiction.
\end{proof}

%\begin{remark}
%Since \(Z(H)\) equals \(Z(C_G(X^*\semi)^\circ)\), we have by
%Lemma \ref{lem:some-centre} that, if \(H\) does not equal \(G\), then
%\((Z(C_G(X^*\semi)^\circ)/Z(G))^\circ\) is non-trivial.
%However, we cannot draw the same conclusion if we know only that
%\(C_G(X^*\semi)^\circ\) does not equal \(G\).
%Indeed, if \(X^*\semi\) is
%the element of \(\mathfrak{sp}_4^*\) defined in \eqref{eq:Xss}, then
%we have that \(C_G(X^*\semi)\) is a copy of \(\SO_4\), but
%\(Z(C_G(X^*\semi))\) equals \(Z(G)\).
%\end{remark}

Conjecture \ref{conj:secondary}\ref{ZXs-vs-ZX-M} is related to
Lemma \ref{lem:CX-vs-CXs}, and
would follow from it if semisimple parts of
Jordan decompositions were unique.
In light of Section \ref{sec:nonuni},
Conjecture \ref{conj:secondary}\ref{strong-uni} is
the strongest possible uniqueness that
one could expect for a Jordan decomposition.

\begin{conjecture}\label{conj:secondary}\hfill
\begin{enumerate}[label=(\roman*)]
\item\label{ZXs-vs-ZX-M}
The group \(C_G(X^*\semi)^\circ\) contains
a maximal torus in \(C_G(X^*)^\circ\).
%If \(T\) is a maximal torus in \(C_G(X^*)_\text{red}^\circ\), then
%\(T \cap Z(C_G(X^*)_\text{red}^\circ)\) is
%contained in \(C_G(X^*\semi)^\circ\).
	%% Less ambitious.
\item\label{strong-uni}
Every semisimple part of \(X^*\) belongs to
the \(C_G(X^*)^\circ(k)\)-orbit of \(X^*\semi\).
\end{enumerate}
\end{conjecture}

We are not able to prove Conjecture \ref{conj:secondary}, but
we can prove a better uniqueness result than
Theorem \ref{thm:main}\ref{uni}.
We state it as Theorem \ref{thm:uni} below.
We need a few ingredients first.

First we prove Lemma \ref{lem:ZXs-vs-ZX-H}, which,
in light of our proof of Theorem \ref{thm:uni}, seems
likely to be a useful ingredient in an eventual proof of
Conjecture \ref{conj:secondary}\ref{strong-uni}.

\begin{lemma}
\label{lem:ZXs-vs-ZX-H}
The multiplicative-type group \(Z\) is contained in \(C_G(X^*)\);
the reduced, connected centralizer
\(C_G(X^*)_\text{\normalfont red}^\circ\) is
contained in \(H\); and,
for every central subgroup \(\mathcal Z\) of \(G\),
the Lie algebra \(\Lie(C_G(X^*)/\mathcal Z)\) is contained in
\(\Lie(H/\mathcal Z)\).
\end{lemma}

\begin{proof}
The image of \(X^*\) in \(\fh^*\) is obviously fixed by \(Z\), so
\(X^*\) itself is fixed by \(Z\), so
\(Z\) is contained in \(C_G(X^*)\).

The natural map
\(\Lie(G) \to \Lie(G/\mathcal Z)\) restricts to an isomorphism of
\(\fh^\perp\) onto
a complement to \(\Lie(H/\mathcal Z)\) in \(\Lie(G/\mathcal Z)\).
Suppose that
\(Y \in \Lie(H/\mathcal Z)\) and \(Y^\perp \in \fh^\perp\) satisfy
\(Y^\perp \ne 0\) and
\(Y + Y^\perp \in \Lie(C_G(X^*)/\mathcal Z)\), so that
\(\ad^*(Y^\perp)(X^*)\) equals \(-\ad^*(Y)(X^*)\).

Let \(\alpha\) be a lowest root
(with respect to the height on \(\Phi(G, T)\) determined by \(B\))
such that the \(T\)-equivariant projection of
\(Y^\perp\) on \(\fg_\alpha\) is
non-zero.
There is a unique element
\(e_{-\alpha}\) of \(\fg_{-\alpha}\) such that
the \(T\)-equivariant projection of
\([Y^\perp, e_{-\alpha}]\) on \(\ft\) is \(h_\alpha\).
Then \([Y^\perp, e_{-\alpha}]\) lies in
\(h_\alpha + \fb\), so
\(\ad^*(Y^\perp)(X^*)(e_{-\alpha}) =
-X^*([Y^\perp, e_{-\alpha}])\) equals
\(-X^*\semi(h_\alpha) \ne 0\).
On the other hand, since
\(-[Y, e_{-\alpha}]\) belongs to \(\fh^\perp\), on which
\(X^*\) is vanishes, we have that
\(-\ad^*(Y)(X^*)(e_{-\alpha}) =
X^*([Y, e_{-\alpha}])\) equals \(0\).
This is a contradiction.

It follows that \(\Lie(C_G(X^*)/\mathcal Z)\) is
contained in \(\Lie(H/\mathcal Z)\) for
all central subgroups \(\mathcal Z\) of \(G\).
In particular,
\(\Lie(C_G(X^*))\) is contained in \(\Lie(H)\).
Now \(C_G(X^*)_\text{red} \cap H\) equals
\(C_{C_G(X^*)_\text{red}}(Z)\), hence is smooth
\cite[Proposition A.8.10(2)]{CGP15}, hence equals
\(C_H(X^*)_\text{red}\).
In particular,
\(\Lie(C_H(X^*)_\text{red}^\circ) =
\Lie(C_H(X^*)_\text{red})\) equals
\(\Lie(C_G(X^*)_\text{red}) \cap H) =
\Lie(C_G(X^*)_\text{red}) \cap \Lie(H)\), which,
by the previous argument, equals
\(\Lie(C_G(X^*)_\text{red})\).
%Since \(C_H(X^*)_\text{red}^\circ\) is smooth and
%\(C_G(X^*)_\text{red}^\circ\) is connected,
%it follows that
%they are equal.
Thus
\(C_H(X^*)_\text{red}^\circ\) equals \(C_G(X^*)_\text{red}^\circ\)
\cite[Proposition 10.15]{Mil17}, so
\(C_G(X^*)_\text{red}^\circ\) is contained in \(H\).
\end{proof}

Since \(H\) equals \(C_G(X^*\semi)\) if and only if
\(\Phi(C_G(X^*\semi), T)\) is closed in \(\Phi(G, T)\),
Theorem \ref{thm:uni} says that
the instances of non-uniqueness of
Jordan decompositions discussed in
Section \ref{sec:nonuni} are the only possibilities.
Since \(\Phi(C_G(X\semi), T)\) is
always closed in \(\Phi(G, T)\) when
\(X\semi \in \fg\) is semisimple, we see that
the Lie-algebra analogue of Theorem \ref{thm:uni} indirectly
asserts the uniqueness of the Jordan decomposition.
Theorem \ref{thm:uni} may thus be viewed as a uniform way of describing
simultaneously
the uniqueness of Jordan decompositions on the Lie algebra, and
the extent of the possible failure of uniqueness on the dual Lie algebra.

\begin{theorem}
\label{thm:uni}
There is a unique closed orbit in the closure of the
\(H\)-orbit of \(X^*\semi\), and it contains
the semisimple part of every Jordan decomposition of \(X^*\).
\end{theorem}

\begin{proof}
We reason by induction on \(\dim(G)\).
The result is obvious if \(G\) is zero-dimensional.
We assume that the result is known for
all groups of smaller dimension than \(G\).

Since restriction to \(H\) provides an
\(H\)-equivariant isomorphism of
the closed, \(H\)-stable subspace \((\fg/\fh^\perp)^*\) of \(\fg^*\) with
\(\fh^*\), the existence of
a unique closed orbit follows from
Theorem \ref{thm:main}\ref{uni}.
We have by \cite[Lemma 3.6]{Spi21} that
\(X^*\semi\) belongs to this unique closed orbit, so
it remains only to show that the semisimple part of
every Jordan decomposition of \(X^*\) is
\(H(k)\)-conjugate to \(X^*\semi\).

Suppose that \(X^{{*}\,\prime}\semi\) is
the semisimple part of
another Jordan decomposition of \(X^*\).
We write \(Z'\) for the schematic centre of
\(C_G(X^{{*}\,\prime}\semi)^\circ\), and put
\(H' = C_G(Z')\).
Lemma \ref{lem:ZXs-vs-ZX-H} gives the containments
\(Z' \subset C_G(X^*)\), hence
\(\Lie(Z'/Z(G)) \subset \Lie(C_G(X^*)/Z(G))\); and
\(\Lie(C_G(X^*)/Z(G)) = \Lie(C_{G/Z(G)}(X^*)) \subset
	\Lie(C_{G/Z(G)}(Z)) = \Lie(H/Z(G))\).
Thus,
\(\Lie((Z' \cap H)/Z(G)) =
\Lie(Z'/Z(G)) \cap \Lie(H/Z(G))\)
equals
\(\Lie(Z'/Z(G))\), which is non-trivial by
Lemma \ref{lem:some-centre}; so
\(Z' \cap H\) is not contained in \(Z(G)\).

By Lemma \ref{lem:CX-vs-CXs}, there is a semisimple part
\(Y^*\semi\) of \(X^*\) that is fixed by
\(\langle Z, Z' \cap H\rangle\).
In particular, \(Y^*\semi\) lies in \((\fg/\fh^\perp)^*\).
Write \(Z_{Y^*\semi}\) for the schematic centre of
\(C_G(Y^*\semi)\).

Since
the restrictions to \(\fh\) of \(X^*\semi\) and \(Y^*\semi\) are
semisimple parts of Jordan decompositions of
the restriction of \(X^*\),
we have by Theorem \ref{thm:main}\ref{uni} that
the restrictions of \(X^*\semi\) and \(Y^*\semi\), hence
\(X^*\semi\) and \(Y^*\semi\) themselves, are
conjugate by an element of \(H(k)\).
In particular, \(Z_{Y^*\semi}\) is \(H(k)\)-conjugate to,
hence equals, \(Z\).

Since \(Z' \cap H\) is not contained in \(Z(G)\), we have that
\(C_G(Z' \cap H)^\circ\) is a proper subgroup of \(G\).
Note that \(C_{C_G(Z' \cap H)^\circ}(Z')^\circ\) equals \(H'\).
Therefore, our inductive result gives that
\(X^{{*}\,\prime}\semi\) and \(Y^*\semi\) are conjugate by
an element of \(H'(k)\).
In particular, \(Z_{Y^*\semi}\) is \(H'(k)\)-conjugate to,
hence equals, \(Z'\).

We therefore have that \(Z\) and \(Z'\) are equal to
\(Z_{Y^*\semi}\), hence to each other;
so \(H\) equals \(H'\); so
\(X^*\semi\), \(Y^*\semi\), and \(X^{{*}\,\prime}\semi\) all
lie in the same \(H(k)\)-orbit, as desired.
\end{proof}

\section{Uniqueness of Jordan decompositions for semisimple and nilpotent elements}
\label{sec:nilp-or-semi-is-nilp-or-semi}

Since we have unqualified uniqueness of Jordan decompositions on \(\fg\),
in this section we discuss only Jordan decompositions on \(\fg^*\).

As we observe in Remark \ref{rem:nilp-is-nilp},
our first uniqueness result
Theorem \ref{thm:main}\ref{uni} implies that
a nilpotent element can have only a trivial Jordan decomposition.
(see Remark \ref{rem:nilp-is-nilp}).
Even our improved uniqueness result Theorem \ref{thm:uni} does not
immediately rule out the possibility that a semisimple elements can have a
non-trivial Jordan decomposition.
This fact would follow from
Conjecture \ref{conj:secondary}\ref{strong-uni}, but
we cannot prove that conjecture.
Nonetheless, we prove in Lemma \ref{lem:semi-is-semi} that
this pathological situation does not arise.

\begin{remark}
\label{rem:nilp-is-nilp}
Theorem \ref{thm:main}\ref{uni} shows that,
if \(X^* \in \fg^*\) is nilpotent, then
the semisimple part of
every Jordan decomposition of \(X^*\) is trivial.
\end{remark}

\begin{lemma}
\label{lem:semi-is-semi}
Suppose that \(X^* \in \fg^*\) is semisimple.  Then
the nilpotent part of
every Jordan decomposition of \(X^*\) is \(0\).
\end{lemma}

\begin{proof}
Let \(X^* = X^{{*}\,\prime}\semi + X^{{*}\,\prime}\nilp\) be a Jordan decomposition.

By \cite[Lemma 3.6]{Spi21}, there is
a cocharacter \(\lambda\) of
\(C_G(X^{{*}\,\prime}\semi)\) such that
\(\lim_{t \to 0} \Ad^*(\lambda(t))X^{{*}\,\prime}\nilp\) equals \(0\), so
\(\lim_{t \to 0} \Ad^*(\lambda(t))X^*\) equals \(X^{{*}\,\prime}\semi\).
Since the orbit of \(X^*\) is closed by
Theorem \ref{thm:main}\ref{key}, we have that
\(X^{{*}\,\prime}\semi\) belongs to \(\Ad^*(G)X^*\).
In particular, the map
\(\mathbb A^1 \setminus \{0\} \ra \Ad^*(G)X^*\)
given by
\(t \mapsto \Ad^*(\lambda(t))X^*\) extends to a map
\(\mathbb A^1 \ra \Ad^*(G)X^*\).
Write \(\mathcal C\) for this image curve.

Let \(T\) be a maximal torus in \(C_G(X^{{*}\,\prime}\semi)\) containing
the image of \(\lambda\).
Suppose that \(X^{{*}\,\prime}\nilp\) does not equal \(0\).
Then there is some
root \(\alpha \in \Phi(C_G(X^{{*}\,\prime}\semi)^\circ, T)\) such that
the \(T\)-equivariant projection \((X^{{*}\,\prime}\nilp)_\alpha\) of \(X^{{*}\,\prime}\nilp\) on
\((\fg^*)_\alpha\) is non-trivial.
We may, and do, choose \(\alpha\) that minimizes
\(\langle\alpha, \lambda\rangle\) among such roots.
Then \(\langle\alpha, \lambda\rangle\) is positive.

By \cite[Lemma 3.1iv)]{KW76},
\(\Lie(C_G(X^{{*}\,\prime}\semi)_\text{red})\) equals \(C_\fg(X^{{*}\,\prime}\semi)\), i.e. $C_G(X^{{*}\,\prime}\semi)$ is smooth. Therefore (as in e.g. \cite[pp. 98, Proposition 6.7]{Bo}) the canonical identification of
the tangent space \(T_{X^{{*}\,\prime}\semi}(\fg^*)\) with
\(\fg^*\) identifies
\(T_{X^{{*}\,\prime}\semi}(\Ad^*(G)X^*)\) with
\(\ad^*(\fg)X^{{*}\,\prime}\semi\). Also
\(T_{X^{{*}\,\prime}\semi}(\mathcal C)\) is identified with
a subspace whose \(T\)-equivariant projection on
\((\fg^*)_\alpha\) includes \((X^{{*}\,\prime}\nilp)_\alpha\), hence
is non-trivial.
However, the \(T\)-equivariant projection of
\(\ad^*(\fg)X^{{*}\,\prime}\semi\) on \((\fg^*)_\alpha\) is trivial,
which is a contradiction.
\end{proof}

Remark \ref{rem:from-full-rank} shows that
we may lift a Jordan decomposition from
the dual Lie algebra of
an arbitrary full-rank subgroup.
Corollary \ref{cor:semi-is-semi} provides a partial converse.

\begin{corollary}
\label{cor:semi-is-semi}
Suppose that \(H\) is a full-rank subgroup of \(G\) such that
\(\fh\) admits an \(H\)-stable complement \(\fh^\perp\) in \(\fg\), and
\(X^* \in \fg^*\) is semisimple, respectively nilpotent, and
trivial on \(\fh^\perp\).
Then the restriction of \(X^*\) to \(\fh\) is
a semisimple, respectively nilpotent, element of \(\fh^*\).
\end{corollary}

\begin{proof}
Write \(X^*_H\) for the restriction of \(X^*\) to \(\fh\), and
let \(X^*_H = X^*_{H\,\semisub} + X^*_{H\,\nilpsub}\) be
a Jordan decomposition.
Then extension of \(X^*_{H\,\semisub}\) and \(X^*_{H\,\nilpsub}\)
trivially across \(\fh^\perp\) to
elements \(X^*\semi\) and \(X^*\nilp\) of \(\fg^*\) gives
a Jordan decomposition \(X^* = X^*\semi + X^*\nilp\)
(by Remark \ref{rem:from-full-rank}), so
Lemma \ref{lem:semi-is-semi} gives that
\(X^*\nilp\) is trivial
if \(X^*\) is semisimple, or
Remark \ref{rem:nilp-is-nilp} gives that
\(X^*\semi\) is trivial
if \(X^*\) is nilpotent.
This means that \(X^*_H = X^*_{H\,\semisub} + X^*_{H\,\nilpsub}\)
equals
\(X^*_{H\,\semisub}\), respectively \(X^*_{H\,\nilpsub}\),
and so is
semisimple, respectively nilpotent.
\end{proof}

\subsection*{Acknowledgments}
We thank Sean Cotner and Stephen DeBacker for very inspiring discussions, and thank David Hansen, Friedrich Knop,
David Benjamin Lim, and Jason Starr for very helpful comments on a
MathOverflow question \cite{MO495402}.
The research of L.~Spice is supported by a gift from the Simons Foundation, award number 636151.
The research of C.-C.~Tsai is supported by NSTC grant 114-2115-M-001-009 and 114-2628-M-001-003.

\p
\bibliographystyle{amsalpha}
\bibliography{biblio.bib}
\end{document}